\documentclass[11pt]{article}
\usepackage{hyperref}
\usepackage{times}  
\usepackage{mathpazo}
\usepackage{amssymb,amsmath,amsthm}
\usepackage{epsfig}

\usepackage{multirow}
\usepackage{array}
\newcolumntype{?}{!{\vrule width 1pt}}
\usepackage{caption}
\usepackage{graphicx}
\usepackage[table]{xcolor}
\usepackage{blkarray}
\usepackage{booktabs}
\usepackage{mathtools}
\usepackage{color}
\definecolor{mygrey}{gray}{0.50}

\newtheorem{theorem}{Theorem}[section]

\newtheorem{definition}[theorem]{Definition}

\newtheorem{lemma}[theorem]{Lemma}

\newtheorem{corollary}[theorem]{Corollary}

\newtheorem{remark}[theorem]{Remark}

\newcommand{\qedsymb}{\hfill{\rule{2mm}{2mm}}}
\renewenvironment{proof}[1][]{\begin{trivlist}
\item[\hspace{\labelsep}{\bf\noindent Proof#1:\/}] }{\qedsymb\end{trivlist}}

\def\calM{{\cal M}}

\def\Z{{\mathbb{Z}}}

\def\mod{\mbox{mod}}

\providecommand{\keywords}[1]{\textbf{\textit{Keywords:~}} #1}
\providecommand{\subjclass}[1]{\textbf{\textit{Classification Codes:~}} #1}

\renewcommand{\epsilon}{\varepsilon}

\newcommand{\rank}{\mathop{\mathrm{rank}}}

\newcommand{\Rbin}{{\rank}_{\mathrm{bin}}}
\newcommand{\Rreal}{{\rank}_\mathbb{R}}

\begin{document}

\title{{\bf The Binary Rank of Circulant Block Matrices}}

\author{
Ishay Haviv\thanks{School of Computer Science, The Academic College of Tel Aviv-Yaffo, Tel Aviv 61083, Israel. Research supported in part by the Israel Science Foundation (grant No.~1218/20).}
\and
Michal Parnas\thanks{School of Computer Science, The Academic College of Tel Aviv-Yaffo, Tel Aviv 61083, Israel. Email address: {\tt michalp@mta.ac.il}}
}

\date{}

\maketitle

\begin{abstract}
The binary rank of a $0,1$ matrix is the smallest size of a partition of its ones into monochromatic combinatorial rectangles.
A matrix $M$ is called $(k_1, \ldots, k_m ; n_1, \ldots, n_m)$ circulant block diagonal if it is a block matrix with $m$ diagonal blocks, such that for each $i \in [m]$, the $i$th diagonal block of $M$ is the circulant matrix whose first row has $k_i$ ones followed by $n_i-k_i$ zeros, and all of whose other entries are zeros.
In this work, we study the binary rank of these matrices as well as the binary rank of their complement, obtained by replacing the zeros by ones and the ones by zeros.
In particular, we compare the binary rank of these matrices to their rank over the reals, which forms a lower bound on the former.

We present a general method for proving upper bounds on the binary rank of block matrices that have diagonal blocks of some specified structure and ones elsewhere.
Using this method, we prove that the binary rank of the complement of a $(k_1, \ldots, k_m ; n_1, \ldots, n_m)$ circulant block diagonal matrix for integers satisfying $n_i > k_i >0$ for each $i \in [m]$ exceeds its real rank by no more than the maximum of $\gcd(n_i,k_i)-1$ over all $i \in [m]$.
We further present several sufficient conditions for the binary rank of these matrices to strictly exceed their real rank.
By combining the upper and lower bounds, we determine the exact binary rank of various families of matrices and, in addition, significantly generalize a result of Gregory (J.~Comb.~ Math.~\&~Comb.~Comp.,~1989).

Motivated by a question of Pullman (Linear Algebra Appl.~,1988), we study the binary rank of $k$-regular $0,1$ matrices, those having precisely $k$ ones in every row and column, and the binary rank of their complement.
As an application of our results on circulant block diagonal matrices, we show that for every $k \geq 2$, there exist $k$-regular $0,1$ matrices whose binary rank is strictly larger than that of their complement.
Furthermore, we exactly determine for every integer $r$, the smallest possible binary rank of the complement of a $2$-regular $0,1$ matrix with binary rank $r$.
\end{abstract}

\newpage

\section{Introduction}

The {\em binary rank} of a $0,1$ matrix $M$ of dimensions $n \times n$, denoted by $\Rbin(M)$, is the minimal integer $r$ for which there exist $0,1$ matrices $M_1$ and $M_2$ of dimensions $n \times r$ and $r \times n$ respectively, such that $M = M_1 \cdot M_2$ with the addition and multiplication operations over the reals.
Equivalently, it is the smallest size of a partition of the ones in $M$ into monochromatic combinatorial rectangles.
From the point of view of graph theory, $\Rbin(M)$ measures the smallest number of bicliques that form a partition of the edges of the bipartite graph associated with $M$ (with $n$ vertices on each part and an edge connecting the $i$th vertex of the first part with the $j$th vertex of the other part whenever $M_{i,j}=1$).

The binary rank has been intensively studied in combinatorics under various equivalent formulations (see, e.g.,~\cite{GregoryPullman,HefnerHLM90,MonsonPR95survey,Schwartz20}), as well as in the area of communication complexity, where it is closely related to the unambiguous non-deterministic communication complexity of functions (see~\cite[Chapter~2]{KN97}).
It can be observed that every $0,1$ matrix $M$ satisfies $\Rbin(M) \geq \Rreal(M)$, where $\Rreal(M)$ stands for the standard rank of the matrix $M$ over the reals.
In contrast to the latter, determining the binary rank of matrices turns out to be a difficult challenge in general.
From a computational perspective, the decision problem associated with the binary rank is known to be NP-hard~\cite{JiangR93}.

A $0,1$ matrix $M$ is said to be {\em $k$-regular} if every row and every column of $M$ has precisely $k$ ones.
Let $\overline{M}$ denote the complement matrix obtained from $M$ by replacing the ones by zeros and the zeros by ones.
In 1986, Brualdi, Manber, and Ross~\cite{BrualdiMR86} proved that for every $k$-regular $0,1$ matrix $M$ of dimensions $n \times n$ with $n > k > 0$, the rank of $M$ over the reals is equal to that of $\overline{M}$, that is, $\Rreal(M) = \Rreal(\overline{M})$.
Following their work, Pullman~\cite{Pullman88} asked in 1988 whether every such matrix $M$ satisfies $\Rbin(M) = \Rbin(\overline{M})$, as holds in the simple case of $k=1$. In 1990, Hefner, Henson, Lundgren, and Maybee~\cite{HefnerHLM90} conjectured that the answer to this question is negative (see~\cite[Conjecture~3.2]{HefnerHLM90}; see also~\cite[Open problem~7.1]{MonsonPR95survey}).
Building on recent advances in communication complexity~\cite{BBGJK21}, the conjecture of~\cite{HefnerHLM90} was settled in~\cite{HavivP22} in a strong form. Namely, it was shown there that for infinitely many integers $r$, there exists a regular matrix $M$ satisfying $\Rbin(M)=r$ and yet $\Rbin(\overline{M}) \leq 2^{\widetilde{O}(\sqrt{\log r})}$. This separation is known to be optimal up to $\log \log r$ multiplicative factors in the exponent.

A prominent class of regular matrices is that of circulant matrices, i.e., square matrices in which every row is a cyclic shift of the row that precedes it by one element to the right. For integers $n \geq k \geq 0$, let $D_{n,k}$ denote the $n \times n$ circulant matrix whose first row starts with $n-k$ ones followed by $k$ zeros.
The real rank of the matrix $D_{n,k}$ is known to satisfy
\begin{eqnarray*}\label{eq:rank_D_n_k}
\Rreal(D_{n,k}) = n-\gcd(n,k)+1
\end{eqnarray*}
for $n > k  \geq 0$ (see Lemma~\ref{lemma:realrank_D}).
Since the real rank forms a lower bound on the binary rank, it follows that $\Rbin(D_{n,k}) = n$ whenever $n$ and $k$ are relatively prime.
For $k=2$, Gregory proved in~\cite{Gregory89}, relying on an algebraic technique of Graham and Pollak~\cite{GrahamP71}, that $\Rbin(D_{n,2})=n$ for all $n \geq 3$.
It is not difficult to see that the binary rank of the complement of $D_{n,2}$, which is equal up to a permutation of columns to the matrix $D_{n,n-2}$, is also $n$ (see Lemma~\ref{lemma:isolation}).
In light of Pullman's question~\cite{Pullman88}, this means that the $2$-regular matrix $D_{n,n-2}$ does not separate the binary rank from the binary rank of the complement.

\subsection{Our Contribution}

In the current work, we aim to study the binary rank of a family of circulant block diagonal matrices and the binary rank of their complement.
This family is defined as follows.

\begin{definition}\label{def:circ_block}
Let $k_1, \ldots, k_m$ and $n_1, \ldots, n_m$ be integers satisfying $n_i \geq k_i \geq 0$ for all $i \in [m]$.
A matrix $M$ is called {\em $(k_1, \ldots, k_m ; n_1, \ldots, n_m)$ circulant block diagonal} if it is a block matrix with $m$ diagonal blocks, such that for each $i \in [m]$, the $i$th diagonal block of $M$ is $D_{n_i,n_i-k_i}$, and all of whose other entries are zeros.
If $k = k_1 = \cdots = k_m$, the matrix $M$ is said to be {\em $(k ; n_1, \ldots, n_m)$ circulant block diagonal}.
\end{definition}
\noindent
Note that every $(k ; n_1, \ldots, n_m)$ circulant block diagonal matrix is $k$-regular.
For the case $k=2$, an old result of Breuer~\cite{Breuer69} asserts that every $2$-regular $0,1$ matrix is equal, up to permutations of rows and columns, to some $(2; n_1, \ldots, n_m)$ circulant block diagonal matrix (see Lemma~\ref{lemma:2-reg-structure}).

Our first contribution is a general method for proving upper bounds on the binary rank of block matrices whose diagonal blocks are of some specified structure and all of whose other entries are ones (see Section~\ref{sec:general_method}). The method is applied to prove the following upper bound on the binary rank of the complement of $(k_1, \ldots, k_m ; n_1, \ldots, n_m)$ circulant block diagonal matrices.
\begin{theorem}\label{thm:upper_new_ki}
Let $M$ be the complement of a $(k_1, \ldots, k_m; n_1, \ldots, n_m)$ circulant block diagonal matrix for integers satisfying $n_i \geq k_i > 0$ for all $i \in [m]$.
For every $i \in [m]$, put $d_i = \gcd(n_i,k_i)$ if $n_i >k_i$, and $d_i = 1$ if $n_i=k_i$.
Then,
\[\Rbin(M) \leq \Rreal(M)+\max_{i \in [m]} d_i-1.\]
\end{theorem}

As a consequence, we produce for every $k \geq 2$, a family of $k$-regular matrices whose binary rank is strictly larger than that of their complement.
\begin{theorem}\label{thm:gap_kIntro}
For all integers $k \geq 2$ and $r \geq 2k$, there exists a $k$-regular $0,1$ matrix $M$ satisfying
\[\Rbin(M) = r~~~~\mbox{and}~~~\Rbin(\overline{M}) \leq \Big \lceil \frac{k+1}{2k} \cdot r\Big \rceil + (2k-3).\]
In particular, for every integer $r \geq 4$, there exists a $2$-regular $0,1$ matrix $M$ satisfying
\[\Rbin(M) = r~~~~\mbox{and}~~~\Rbin(\overline{M}) \leq \Big \lceil \frac{3r}{4}\Big \rceil + 1.\]
\end{theorem}
\noindent
Theorem~\ref{thm:gap_kIntro} shows that the answer to the aforementioned question of Pullman~\cite{Pullman88} remains negative even when it is restricted to $2$-regular matrices. Note that this is in a strong contrast to the matrices provided in~\cite{HavivP22} whose regularity grows with their dimensions.
We further mention that the construction given in the proof of Theorem~\ref{thm:gap_kIntro} is elementary and is from the family of circulant block diagonal matrices, whereas the construction of~\cite{HavivP22} relies on powerful tools from communication complexity (see~\cite{BBGJK21}).

By combining Theorem~\ref{thm:upper_new_ki} with a result of~\cite{Breuer69}, we determine the binary rank of the complement of any $2$-regular matrix up to an additive $1$, as stated below.
\begin{theorem}\label{thm:2-regIntro}
Let $M$ be a $2$-regular $0,1$ matrix of dimensions $n \times n$.
Then, $M$ is equal, up to permutations of rows and columns, to a $(2;n_1, \ldots,n_m)$ circulant block diagonal matrix for some integers $m \geq 1$ and $n_1,\ldots,n_m \geq 2$. Let $m_2$ denote the number of indices $i \in [m]$ with $n_i=2$, and let $m_{even}$ denote the number of indices $i \in [m]$ for which $n_i$ is even.
Unless $M$ is the all-one $2 \times 2$ matrix, it holds that
\[\Rbin(M) = n-m_2~~~~\mbox{and}~~~~\Rbin(\overline{M}) \in \{n-m_{even}, n-m_{even}+1\}.\]
\end{theorem}

We now turn to present our lower bounds.
For certain families of $2$-regular matrices $M$, we show that the upper bound of $n-m_{even}+1$ on $\Rbin(\overline{M})$ given in Theorem~\ref{thm:2-regIntro} is tight (see Theorems~\ref{thm:2_even} and~\ref{thm:2_even_odd}).
This is used to prove the following lower bound on the binary rank of the complement of any $2$-regular matrix with a given binary rank $r$.
\begin{theorem}\label{thm:2-regular}
For every integer $r \geq 4$ and for every $2$-regular $0,1$ matrix $M$ with $\Rbin(M) = r$, the binary rank of its complement satisfies
\[\Rbin(\overline{M}) \geq \Big \lceil\frac{3r}{4} \Big \rceil+1.\]
\end{theorem}
\noindent
Note that Theorem~\ref{thm:gap_kIntro} implies that the bound given in Theorem~\ref{thm:2-regular} is tight (see Corollary~\ref{cor:2-regular}).

The proof of Theorem~\ref{thm:2-regular} relies on some general lower bounds we provide on the binary rank of the complement of $(k_1, \ldots, k_m; n_1, \ldots, n_m)$ circulant block diagonal matrices. In this context, we are particularly interested in identifying the cases for which the binary rank of those matrices exceeds their real rank.
We provide a variety of sufficient conditions for such a statement to hold.
One of them is given by the following theorem.
\begin{theorem}\label{thm:ki_divides_ni+1}
Let $M$ be the complement of a $(k_1, \ldots, k_m; n_1, \ldots, n_m)$ circulant block diagonal matrix for integers satisfying $n_i \geq k_i > 0$ for all $i \in [m]$.
Suppose that for each $i \in [m]$, $k_i$ divides $n_i$, and that some $i \in [m]$ satisfies $n_i > k_i >1$.
Then,
\[\Rbin(M) \geq \Rreal(M) + 1.\]
\end{theorem}

For the complement of a $(k; n_1, \ldots, n_m)$ circulant block diagonal matrix, the condition in Theorem~\ref{thm:ki_divides_ni+1} requires $k$ to divide $n_i$ for all $i \in [m]$.
However, we prove that the desired lower bound holds even under the weaker condition that some integer $d>1$ satisfies $d = \gcd(n_i,k)$ for all $i \in [m]$ (see Theorem~\ref{thm:common_k_d_same}). As a consequence, we derive the following result for the matrices $D_{n,k}$, thus extending a result of~\cite{Gregory89} which concerns the case of $k=2$.
\begin{theorem}\label{thm:D_n_k}
For all integers $n > k >0$, $\Rbin(D_{n,k}) \geq \min(\Rreal(D_{n,k})+1,n)$.
\end{theorem}
\noindent
We refer the reader to Theorems~\ref{thm:ni=ki+di},~\ref{thm:common_k_d}, and~\ref{thm:prime} for additional sufficient conditions for the binary rank of the complement of a $(k_1, \ldots, k_m; n_1, \ldots, n_m)$ circulant block diagonal matrix to exceed its real rank.

We conclude this section with a few questions raised by our work.
It would be interesting to better understand the binary rank of the complement of $(k_1, \ldots, k_m ; n_1, \ldots, n_m)$ circulant block diagonal matrices.
In particular, it would be nice to figure out when the upper bound provided by Theorem~\ref{thm:upper_new_ki} is tight.
While we have shown that this bound is tight in several cases, we are not aware of any choice of the parameters for which it is not.
Of special interest is the problem of determining the binary rank of the single-block matrices $D_{n,k}$ for $k \geq 3$.
Note that for $k=3$, a conjecture of Hefner et al.~\cite{HefnerHLM90} asserts that $\Rbin(D_{n,3})=n$ for all $n \geq 6$.
Finally, it would be interesting to extend the statement of Theorem~\ref{thm:2-regular} and to determine for all integers $k$ and $r$, the smallest possible binary rank of the complement of a $k$-regular matrix with binary rank $r$.

\subsection{Outline}
The rest of the paper is organized as follows.
In Section~\ref{sec:prem}, we gather several definitions and results that will be used throughout the paper.
In Section~\ref{sec:upper}, we present our general method for proving upper bounds on the binary rank of block matrices and derive Theorems~\ref{thm:upper_new_ki},~\ref{thm:gap_kIntro}, and~\ref{thm:2-regIntro}.
In Section~\ref{sec:lower}, we provide sufficient conditions for the binary rank of the complement of a $(k_1, \ldots, k_m ; n_1, \ldots, n_m)$ circulant block diagonal matrix to exceed its real rank, and confirm Theorems~\ref{thm:ki_divides_ni+1} and~\ref{thm:D_n_k}.
Finally, in Section~\ref{sec:2-reg}, we focus on the binary rank of the complement of $2$-regular matrices and prove Theorem~\ref{thm:2-regular}.

\section{Preliminaries}\label{sec:prem}

\subsection{The Matrix $D_{n,k}$}

A circulant matrix is a square matrix in which every row is obtained by a cyclic shift of the row that precedes it by one element to the right.
Recall that for integers $n \geq k \geq 0$, $D_{n,k}$ is the $n \times n$ circulant matrix whose first row starts with $n-k$ ones followed by $k$ zeros.
It will be convenient here to state the results on these matrices in terms of the $k$-regular matrix $D_{n,n-k}$, whose first row has $k$ ones followed by $n-k$ zeros.
Its real rank is determined by the following lemma, which stems from a result of Ingleton~\cite{ingleton1956rank} (see, e.g.,~\cite{GradyN97}).
\begin{lemma}\label{lemma:realrank_D}
For all integers $n \geq k >0$, $\Rreal(D_{n,n-k}) =n - \gcd(n,k) + 1$.
\end{lemma}

We next observe that the binary rank of $D_{n,n-k}$ is full, provided that the number $k$ of ones in a row does not exceed $\lceil n/2 \rceil$.
\begin{lemma}\label{lemma:isolation}
For all integers $n$ and $k$ such that $0<k \leq \lceil n/2 \rceil$, $\Rbin(D_{n,n-k}) =n$.
\end{lemma}

\begin{proof}
It clearly holds that $\Rbin(D_{n,n-k}) \leq n$, because the ones of every row can be covered by a single combinatorial rectangle.
For the other direction, consider the $n$ diagonal entries of $D_{n,n-k}$.
By $k > 0$, they are all ones.
Using $k \leq \lceil n/2 \rceil$, it can be verified that they form an isolation set, that is, no two of them belong to an all-one $2 \times 2$ matrix. Therefore, no two of them can belong to a common combinatorial rectangle of ones. This implies that $\Rbin(D_{n,n-k}) \geq n$ and completes the proof.
\end{proof}

\subsection{Block Matrices}

\subsubsection{Notation}
For integers $m \geq 1$ and $n_1, \ldots, n_m \geq 1$, let $M$ be a block matrix with $m$ diagonal blocks, where for each $i \in [m]$, the $i$th diagonal block of $M$ is of dimensions $n_i \times n_i$.
Throughout the work, we denote the index of a row or a column of such a matrix $M$ by a pair $(i,j)$ with $i \in [m]$ and $j \in [n_i]$, referring to the $j$th row or column of the $i$th diagonal block of $M$.
We refer to the sub-matrix of $M$ that consists of the rows of $\{i\} \times [n_i]$ and the columns of $\{j\} \times [n_j]$ as the $(i,j)$ block of $M$.
Note that the $(i,i)$ block of $M$ is its $i$th diagonal block.

\subsubsection{Regular Diagonal Blocks}

It was shown in~\cite{BrualdiMR86} that for every $k$-regular $0,1$ matrix $M$ of dimensions $n \times n$ with $n > k > 0$, it holds that $\Rreal(M) = \Rreal(\overline{M})$. The following lemma generalizes the result of~\cite{BrualdiMR86} to block diagonal matrices whose diagonal blocks are regular.

\begin{lemma}\label{lemma:rank_ki}
Let $M$ be a block matrix with $m$ diagonal blocks, where for each $i \in [m]$, the $i$th diagonal block of $M$ is a $k_i$-regular $0,1$ matrix of dimensions $n_i \times n_i$ for integers $n_i \geq k_i>0$, and all of whose other entries are zeros. Then, $\Rreal(M) = \Rreal(\overline{M})$, unless $M$ is an all-one matrix.
\end{lemma}

For the proof of Lemma~\ref{lemma:rank_ki}, we need the following lemma attributed to Ryser in~\cite{BrualdiMR86}.
\begin{lemma}[{\cite[Lemma~2.1]{BrualdiMR86}}]\label{lemma:Ingle}
For every $0,1$ matrix $M$, the following two statements are equivalent.
\begin{enumerate}
  \item $\Rreal(M) = \Rreal(\overline{M})$.
  \item The all-one vector belongs to the row span of $M$ if and only if it belongs to the row span of $\overline{M}$.
\end{enumerate}
\end{lemma}

\begin{proof}[ of Lemma~\ref{lemma:rank_ki}]
Let $M$ be a matrix as in the statement of the lemma. It can be assumed that $m \geq 2$ or $n_1 > k_1$, because otherwise $M$ is an all-one matrix.
For every $i \in [m]$, the $i$th diagonal block of $M$ is $k_i$-regular, so in particular, it has $k_i$ ones in each of its columns.
Consider the vector obtained by summing the rows of $M$ that correspond to the $i$th block, and observe that it has the value $k_i$ in the entries of the $i$th block and zeros elsewhere. By considering the linear combination of the rows of $M$, where for each $i \in [m]$, the coefficient of every row of the $i$th block is $1/k_i$, it follows that the all-one vector belongs to the row span of $M$.

Consider next the linear combination of the rows of $\overline{M}$, where for each $i \in [m]$, the coefficient of every row of the $i$th block is $a_i = \prod_{j \in [m] \setminus \{i\}}{k_j}$. Observe that for each $i \in [m]$, the value of the obtained vector in the entries of the $i$th block is
\[ a_i  \cdot (n_i-k_i) + \sum_{j \in [m] \setminus \{i\}}{a_j \cdot n_j} = \sum_{j=1}^{m}{a_j \cdot n_j}- a_i \cdot k_i = \sum_{j=1}^{m}{a_j \cdot n_j}-\prod_{j=1}^{m}{k_j}.\]
Notice that the latter is independent of $i$. Moreover, it is nonzero, because $a_i >0$ for all $i \in [m]$ and, in addition, $m \geq 2$ or $n_1 > k_1$. It thus follows that the all-one vector belongs to the row span of $\overline{M}$. Since the all-one vector belongs to the row span of both $M$ and $\overline{M}$, Lemma~\ref{lemma:Ingle} implies that $\Rreal(M) = \Rreal(\overline{M})$, as required.
\end{proof}

\subsubsection{Circulant Block Diagonal Matrices}

Recall that a matrix is called $(k_1, \ldots, k_m ; n_1, \ldots, n_m)$ circulant block diagonal if it is a block matrix with $m$ diagonal blocks, such that for each $i \in [m]$, its $i$th diagonal block is $D_{n_i,n_i-k_i}$, and all of whose other entries are zeros (see Definition~\ref{def:circ_block}).
By combining Lemmas~\ref{lemma:realrank_D} and~\ref{lemma:rank_ki}, we determine the real rank of such matrices and of their complement.

\begin{lemma}\label{lemma:rank_M_comp}
Let $M$ be a $(k_1, \ldots, k_m; n_1, \ldots, n_m)$ circulant block diagonal matrix for integers satisfying $n_i \geq k_i > 0$ for all $i \in [m]$.
If $M$ is not an all-one matrix, then
\[\Rreal(M) = \Rreal(\overline{M}) = \sum_{i=1}^{m}{(n_i-\gcd(n_i,k_i)+1)}.\]
\end{lemma}

\begin{proof}
By Definition~\ref{def:circ_block}, for each $i \in [m]$ it holds that the $i$th diagonal block of $M$ is equal to $D_{n_i,n_i-k_i}$, and is thus $k_i$-regular, and all the other entries of $M$ are zeros.
Since $M$ is not an all-one matrix, using the assumption $n_i \geq k_i > 0$ for all $i \in [m]$, we can apply Lemma~\ref{lemma:rank_ki} to obtain that
\begin{eqnarray*}
\Rreal(M) = \Rreal(\overline{M}).
\end{eqnarray*}
Further, since $M$ is block diagonal, its real rank is the sum of the real rank of its diagonal blocks.
Using Lemma~\ref{lemma:realrank_D}, we obtain that
\[ \Rreal(M) = \sum_{i=1}^{m}{\Rreal(D_{n_i,n_i-k_i})} = \sum_{i=1}^{m}{(n_i-\gcd(n_i,k_i)+1)},\]
and we are done.
\end{proof}

We derive the following lemma.
\begin{lemma}\label{lemma:2-reg-rankR}
Let $M$ be a $(2; n_1, \ldots, n_m)$ circulant block diagonal matrix for integers satisfying $n_i \geq 2$ for all $i \in [m]$, and put $n = \sum_{i=1}^{m}{n_i}$.
Let $m_{even}$ denote the number of indices $i \in [m]$ for which $n_i$ is even.
If $M$ is not the all-one $2 \times 2$ matrix, then
\[\Rreal(M) = \Rreal(\overline{M}) = n-m_{even}.\]
\end{lemma}
\begin{proof}
Suppose that $M$ is not the all-one $2 \times 2$ matrix.
By Lemma~\ref{lemma:rank_M_comp}, we obtain that
\[\Rreal(M) = \Rreal(\overline{M}) = \sum_{i=1}^{m}{(n_i - \gcd(n_i,2)+1)} = n - m_{even},\]
so we are done.
\end{proof}

\subsubsection{A Characterization of $2$-Regular Matrices}

The following simple lemma characterizes the structure of $2$-regular $0,1$ matrices.
Its proof uses an argument of Breuer~\cite{Breuer69}, presented here essentially for completeness.
In what follows, a matrix is said to have a {\em single block} if it is impossible to apply permutations to its rows and columns and to obtain a block matrix with two square diagonal blocks, where all the other entries are zeros.
\begin{lemma}\label{lemma:2-reg-structure}
Every $2$-regular $0,1$ matrix is equal, up to permutations of the rows and columns, to a $(2; n_1, \ldots, n_m)$ circulant block diagonal matrix for some integers $m \geq 1$ and $n_1, \ldots, n_m \geq 2$.
\end{lemma}

\begin{proof}
Let $M$ be a $2$-regular $0,1$ matrix.
We first claim that $M$ is equal, up to permutations of rows and columns, to a block matrix with square diagonal blocks, such that each diagonal block has a single block and all the other entries are zeros.
Indeed, this trivially holds in case that $M$ has a single block.
Otherwise, it is possible to apply permutations to its rows and columns to obtain a block matrix with two square diagonal blocks, where all the other entries are zeros.
By repeatedly applying this procedure to the obtained diagonal blocks, we get a matrix with the desired form.
Let $m \geq 1$ denote the number of its diagonal blocks, and for each $i \in [m]$, let $n_i$ denote the number of rows (or columns) in the $i$th diagonal block.

Since $M$ is $2$-regular, it follows that each of the above diagonal blocks is $2$-regular as well, and thus $n_i \geq 2$ for all $i \in [m]$.
We claim that the $i$th diagonal block is equal, up to permutations of its rows and columns, to the matrix $D_{n_i,n_i-2}$.
Indeed, if $n_i=2$, then the $i$th diagonal block is $D_{2,0}$.
Otherwise, using the $2$-regularity, one can apply a suitable permutation of columns to obtain $(1,1,0\ldots,0)$ as the first row of the block.
Using again the $2$-regularity and the fact that $n_i >2$, one can apply suitable permutations of rows and columns to obtain $(0,1,1,0,\ldots,0)$ as the second row of the block.
Proceeding this way, it follows that by applying some permutations to the rows and columns of the $i$th diagonal block, the $j$th row of the block, for $j \in [n_i]$, has ones in columns $j$ and $j+1$ of the block and zeros elsewhere (where $n_i+1$ is identified with $1$).
This implies that the $i$th diagonal block of the obtained matrix is $D_{n_i,n_i-2}$ for every $i \in [m]$. Hence, it is a $(2; n_1, \ldots, n_m)$ circulant block diagonal matrix for integers satisfying $n_i \geq 2$ for all $i \in [m]$, and we are done.
\end{proof}

\section{Upper Bounds}\label{sec:upper}

In this section, we present a general method for proving upper bounds on the binary rank of block matrices with diagonal blocks of some specified structure and ones elsewhere.
We then apply this method to the complement of circulant block diagonal matrices and prove Theorem~\ref{thm:upper_new_ki}.

\subsection{The General Method}\label{sec:general_method}

We start with the following definition.
\begin{definition}\label{def:M_t,r}
For integers $t > r \geq 0$, let $\calM_{t,r}$ be the collection of all $0,1$ matrices $M$ of dimensions $n \times n$ for an integer $n \geq 1$ for which there exist sets $A_1, \ldots, A_t \subseteq [n]$ and $B_1, \ldots, B_t \subseteq [n]$ such that
\begin{enumerate}
  \item\label{itm:M1} the combinatorial rectangles $A_1 \times B_1, \ldots, A_t \times B_t$ form a partition of the ones in $M$,
  \item\label{itm:M2} the sets $A_1, \ldots, A_r$ are pairwise disjoint,
  \item\label{itm:M3} there exists a set $L \subseteq [t] \setminus [r]$ for which the sets $B_l$ with $l \in L$ form a partition of $[n]$, and
  \item\label{itm:M4} for every $s \in [r]$, there exists a set $L_s \subseteq [t] \setminus [r]$ for which the sets $B_l$ with $l \in L_s$ form a partition of $[n] \setminus B_s$.
\end{enumerate}
\end{definition}

Our method supplies an upper bound on the binary rank of a block matrix $M$ that has diagonal blocks from the families $\calM_{t,r}$ and ones elsewhere.
Before the formal statement and proof, let us briefly describe the high-level idea of our construction.
For concreteness, suppose that $M$ has $m$ diagonal blocks, and that for every $i \in [m]$, the $i$th diagonal block of $M$ belongs to the family $\calM_{t_i,r_i}$ for the integers $t_i > r_i \geq 0$.

First, it is easy to see that the binary rank of $M$ is at most $\sum_{i=1}^{m}{t_i}$.
Indeed, for any $i \in [m]$, Item~\ref{itm:M1} of Definition~\ref{def:M_t,r} implies that there exist $t_i$ combinatorial rectangles that form a partition of the ones in the $i$th diagonal block of $M$.
Moreover, Item~\ref{itm:M3} of Definition~\ref{def:M_t,r} guarantees that the set of columns of the $i$th block can be represented as a disjoint union of column sets of these rectangles. This allows us to extend the sets of rows of these $t_i$ rectangles so that they will also cover the ones above and below the $i$th diagonal block. By applying this procedure to all diagonal blocks, we get a partition of the ones in the matrix $M$ into $\sum_{i=1}^{m}{t_i}$ combinatorial rectangles, yielding the claimed bound on its binary rank.

To obtain a more economical partition of the ones in $M$, we use the other properties of the partitions of the diagonal blocks of $M$ given by Definition~\ref{def:M_t,r}.
Specifically, the definition guarantees that among the $t_i$ rectangles of the $i$th diagonal block of $M$, there are $r_i$ special rectangles, those labelled $1, \ldots, r_i$.
We show, roughly speaking, that the special rectangles of the $m$ diagonal blocks can be combined together into fewer rectangles which can be completed to a relatively small partition of the ones in $M$.

To do so, we first consider some set of special rectangles of different diagonal blocks of $M$, and construct a rectangle whose rows and columns consist of all the rows and columns of these rectangles. Next, we consider a set of some other special rectangles of different diagonal blocks of $M$, and again construct a rectangle whose rows and columns consist of the rows and columns of all of them. Proceeding this way, we obtain a collection of merged rectangles, whose number is the largest $r_i$ with $i \in [m]$. We further add to this collection the other $t_i-r_i$ non-special rectangles of every block $i \in [m]$. It turns out that the properties of the rectangles from Definition~\ref{def:M_t,r} allow us to extend the row sets of the non-special rectangles of each block to cover the remaining uncovered ones in its columns without increasing the number of rectangles. This results in a partition of the ones in $M$ into only $\sum_{i=1}^{m}{(t_i-r_i)} + \max_{i \in [m]}{r_i}$ combinatorial rectangles.

We now present the formal statement and the full proof (see Figure~\ref{fig:upperbound} for an illustration).

\begin{theorem}\label{thm:upper_M_t,r}
For an integer $m$, let $t_1, \ldots, t_m$ and $r_1, \ldots, r_m$ be integers satisfying $t_i > r_i \geq 0$ for each $i \in [m]$.
Let $M$ be a block matrix with $m$ diagonal blocks, such that for each $i \in [m]$, the $i$th diagonal block of $M$ belongs to $\calM_{t_i,r_i}$, and all of whose other entries are ones.
Then,
\[\Rbin(M) \leq \sum_{i =1}^{m}{(t_i-r_i)} + \max_{i \in [m]}{r_i}.\]
\end{theorem}

\begin{proof}
For each $i \in [m]$, consider the $i$th diagonal block $M^{(i)}$ of $M$, and let $n_i$ denote the number of its rows (or columns). Put $n = \sum_{i \in [m]}{n_i}$. By assumption, the matrix $M^{(i)}$ belongs to $\calM_{t_i,r_i}$.
Considering the indices of $\{i\} \times [n_i]$ for the rows and columns of $M^{(i)}$, let $A^{(i)}_l, B^{(i)}_l \subseteq \{i\} \times [n_i]$ for $l \in [t_i]$ be the sets that satisfy the conditions given in Definition~\ref{def:M_t,r} with respect to $M^{(i)}$.
It follows, using Item~\ref{itm:M1} of Definition~\ref{def:M_t,r}, that the rectangles $R^{(i)}_l = A^{(i)}_l \times B^{(i)}_l$ for all $i \in [m]$ and $l \in [t_i]$ do not overlap and cover the ones of the diagonal blocks of $M$.
Assume without loss of generality that $r_1 = \max_{i \in [m]}{r_i}$.

The matrix $M$ can be viewed as a concatenation of $m$ sub-matrices $Q_1, \ldots, Q_{m}$, where $Q_i$ is the $n \times n_i$ matrix that consists of the columns of $M$ that correspond to the $i$th block.
Note that $Q_i$ consists of the matrix $M^{(i)}$ in the rows of $\{i\} \times [n_i]$ and ones elsewhere, and that its columns are indexed by $\{i\} \times [n_i]$.

We turn to describe a partition of the ones in $M$ into combinatorial rectangles.
We first define a collection of $r_1$ rectangles based on the rectangles $R^{(i)}_s$ with $i \in [m]$ and $s \in [r_i]$.
To do so, for every $s \in [r_1]$, consider the set $F_s$ of all blocks $i \in [m]$ whose $s$th rectangle $R^{(i)}_s$ exists and is among the first $r_i$ rectangles of the block, namely, $F_s = \{ i \in [m] \mid s \leq r_i \}$.
For every $s \in [r_1]$, we define the rectangle
\[P_s = \Big ( \cup_{i \in F_s}{A^{(i)}_s} \Big ) \times \Big ( \cup_{j \in F_s}{B^{(j)}_s} \Big ),\]
and we let $P$ denote the collection of these $r_1$ rectangles.
We claim that for every $s \in [r_1]$, $P_s$ is a rectangle of ones in $M$.
Indeed, for every $i, j \in F_s$, the rectangle $P_s$ covers in block $(i,j)$ of $M$ the ones covered by the rectangle $A^{(i)}_s \times B^{(j)}_s$.
For $i=j$, this rectangle belongs to the given partition of the ones in $M^{(i)}$, and for $i \neq j$, all the entries of block $(i,j)$ in $M$ are ones, hence for both cases, the entries of $P_s$ have ones in $M$.
We further claim that the rectangles of $P$ do not overlap. This indeed follows from Item~\ref{itm:M2} of Definition~\ref{def:M_t,r} that guarantees that their sets of rows are pairwise disjoint.

We turn now to cover the remaining uncovered ones in $M$.
To do so, fix some $i \in [m]$, and let us describe a collection of rectangles that form a partition of the remaining uncovered ones in the sub-matrix $Q_i$.
First, we consider the $t_i-r_i$ rectangles $R^{(i)}_l$ for $l \in [t_i] \setminus [r_i]$ whose support lies on the $i$th diagonal block of $M$. Recall that these rectangles, together with the rectangles $R^{(i)}_s$ for $s \in [r_i]$ that are contained in the rectangles of $P$, form by Item~\ref{itm:M1} of Definition~\ref{def:M_t,r} a partition of the ones in the rows of $\{i\} \times [n_i]$ in $Q_i$.
Each of the other rows of $Q_i$, whose indices do not belong to $\{i\} \times [n_i]$, is either uncovered at all so far, or is covered by the rectangles of $P$ at the columns of a unique set $B^{(i)}_s$ with $s \in [r_i]$, where the uniqueness follows from Item~\ref{itm:M2} of Definition~\ref{def:M_t,r}.
By Item~\ref{itm:M3} of Definition~\ref{def:M_t,r}, the set of columns $\{i \} \times [n_i]$ of $Q_i$ is a disjoint union of certain column sets $B^{(i)}_l$ with $l \in [t_i] \setminus [r_i]$. This implies that the all-one rows in $Q_i$ that are not covered at all so far, can be covered by adding them to the row sets of the corresponding rectangles $R^{(i)}_l$ with $l \in [t_i] \setminus [r_i]$.
Similarly, by Item~\ref{itm:M4} of Definition~\ref{def:M_t,r}, for every $s \in [r_i]$ it holds that the set of columns $(\{i \} \times [n_i]) \setminus B^{(i)}_s$ of $Q_i$ is a disjoint union of certain column sets $B^{(i)}_l$ with $l \in [t_i] \setminus [r_i]$. Therefore, the sets of rows of the rectangles $R^{(i)}_l$ with $l \in [t_i] \setminus [r_i]$ can be extended to cover the remaining uncovered ones in $Q_i$ as well.

By applying the construction described above to all diagonal blocks $i \in [m]$ and by combining the produced rectangles with those of $P$, we obtain a collection of $\sum_{i =1}^{m}{(t_i-r_i)} + r_1$ rectangles that form a partition of the ones in $M$. This yields the desired upper bound on the binary rank of $M$, so we are done.
\end{proof}

\newcolumntype{?}{!{\vrule width 2pt}}
\newcolumntype{d}{>{\hfil$}p{.3cm}<{$\hfil}}

\begin{figure}[htb!]
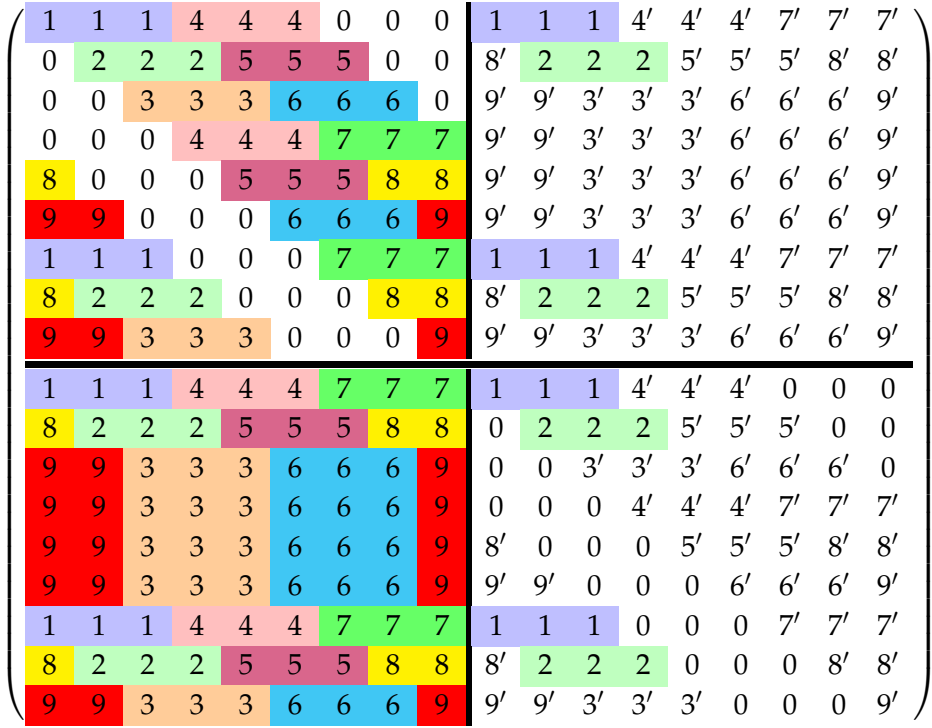

\captionsetup{width=0.9\textwidth}
$$
\begin{array}{ccc} &
\left( \begin{array}{ddddddddd?  ddddddddd}
\cellcolor{blue!25} 1 & \cellcolor{blue!25} 1  & \cellcolor{blue!25} 1  & \cellcolor{red!25}4 &  \cellcolor{red!25}4 &  \cellcolor{red!25}4 & 0 & 0 & 0 &  \cellcolor{blue!25} 1 & \cellcolor{blue!25} 1  & \cellcolor{blue!25} 1  & 4' & 4' & 4' & 7' & 7' &  7' \\
0 &  \cellcolor{green!25} 2 & \cellcolor{green!25} 2 & \cellcolor{green!25} 2 &\cellcolor{purple!60} 5 & \cellcolor{purple!60} 5 & \cellcolor{purple!60} 5 & 0 & 0 & 8' &  \cellcolor{green!25} 2 & \cellcolor{green!25} 2 & \cellcolor{green!25} 2 & 5' & 5' & 5' & 8' & 8'  \\
0 & 0 & \cellcolor{orange!40}3 & \cellcolor{orange!40}3 & \cellcolor{orange!40}3 & \cellcolor{cyan!60} 6 & \cellcolor{cyan!60} 6 & \cellcolor{cyan!60} 6 & 0 & 9' & 9' & 3' & 3' & 3' & 6' & 6' & 6' & 9' \\
0 & 0 & 0 & \cellcolor{red!25}4 &  \cellcolor{red!25}4 &  \cellcolor{red!25}4 & \cellcolor{green!60}7 & \cellcolor{green!60}7 & \cellcolor{green!60}7 & 9' & 9' & 3' & 3' & 3' & 6' & 6' & 6' & 9'  \\
\cellcolor{yellow}8 & 0 & 0 & 0 & \cellcolor{purple!60} 5 & \cellcolor{purple!60} 5 & \cellcolor{purple!60} 5 & \cellcolor{yellow}8 & \cellcolor{yellow}8 & 9' & 9' & 3' & 3' & 3' & 6' & 6' & 6' & 9' \\
\cellcolor{red} 9 &  \cellcolor{red} 9 & 0 & 0 & 0 & \cellcolor{cyan!60} 6 & \cellcolor{cyan!60} 6 & \cellcolor{cyan!60} 6 &  \cellcolor{red} 9 & 9' & 9' & 3' & 3' & 3' & 6' & 6' & 6' & 9'  \\
\cellcolor{blue!25} 1 & \cellcolor{blue!25} 1  & \cellcolor{blue!25} 1  & 0 & 0 & 0 & \cellcolor{green!60}7 & \cellcolor{green!60}7 & \cellcolor{green!60}7 & \cellcolor{blue!25} 1 & \cellcolor{blue!25} 1  & \cellcolor{blue!25} 1  & 4' & 4' & 4' & 7' & 7' & 7' \\
\cellcolor{yellow}8 &  \cellcolor{green!25} 2 & \cellcolor{green!25} 2 & \cellcolor{green!25} 2 & 0 & 0 & 0 & \cellcolor{yellow}8 & \cellcolor{yellow}8 & 8' & \cellcolor{green!25} 2 & \cellcolor{green!25} 2 & \cellcolor{green!25} 2 & 5' & 5' & 5' & 8' & 8'  \\
\cellcolor{red} 9 & \cellcolor{red} 9 & \cellcolor{orange!40}3 & \cellcolor{orange!40}3 & \cellcolor{orange!40}3  & 0 & 0 & 0 & \cellcolor{red} 9 & 9' & 9' & 3' & 3' & 3' & 6' & 6' & 6' & 9'  \\
\specialrule{.2em}{.1em}{.1em}
\cellcolor{blue!25} 1 & \cellcolor{blue!25} 1  & \cellcolor{blue!25} 1  & \cellcolor{red!25}4 &  \cellcolor{red!25}4 &  \cellcolor{red!25}4 & \cellcolor{green!60}7 & \cellcolor{green!60}7 & \cellcolor{green!60}7 & \cellcolor{blue!25} 1 & \cellcolor{blue!25} 1  & \cellcolor{blue!25} 1  & 4' & 4' & 4' & 0 & 0 & 0   \\
\cellcolor{yellow}8 &  \cellcolor{green!25} 2 & \cellcolor{green!25} 2 & \cellcolor{green!25} 2 & \cellcolor{purple!60} 5 & \cellcolor{purple!60} 5 & \cellcolor{purple!60} 5 & \cellcolor{yellow}8 & \cellcolor{yellow}8 &  0 &  \cellcolor{green!25} 2 & \cellcolor{green!25} 2 & \cellcolor{green!25} 2 & 5' & 5' & 5' & 0 & 0   \\
\cellcolor{red} 9 & \cellcolor{red} 9  & \cellcolor{orange!40}3 & \cellcolor{orange!40}3 & \cellcolor{orange!40}3  & \cellcolor{cyan!60} 6 & \cellcolor{cyan!60} 6 & \cellcolor{cyan!60} 6 & \cellcolor{red} 9 & 0 & 0 & 3' & 3' & 3' & 6' & 6' & 6' & 0  \\
\cellcolor{red} 9 & \cellcolor{red} 9  & \cellcolor{orange!40}3 & \cellcolor{orange!40}3 & \cellcolor{orange!40}3  & \cellcolor{cyan!60} 6 & \cellcolor{cyan!60} 6 & \cellcolor{cyan!60} 6 & \cellcolor{red} 9 &0 & 0 & 0 & 4' & 4' & 4' & 7' & 7' & 7'   \\
\cellcolor{red} 9 & \cellcolor{red} 9  & \cellcolor{orange!40}3 & \cellcolor{orange!40}3 & \cellcolor{orange!40}3  & \cellcolor{cyan!60} 6 & \cellcolor{cyan!60} 6 & \cellcolor{cyan!60} 6 & \cellcolor{red} 9 &  8' & 0 & 0 & 0 & 5' & 5' & 5' & 8' & 8'  \\
\cellcolor{red} 9 & \cellcolor{red} 9  & \cellcolor{orange!40}3 & \cellcolor{orange!40}3 & \cellcolor{orange!40}3  & \cellcolor{cyan!60} 6 & \cellcolor{cyan!60} 6 & \cellcolor{cyan!60} 6& \cellcolor{red} 9 & 9' & 9' & 0 & 0 & 0 & 6' & 6' & 6' & 9'   \\
\cellcolor{blue!25} 1 & \cellcolor{blue!25} 1  & \cellcolor{blue!25} 1  & \cellcolor{red!25}4 &  \cellcolor{red!25}4 &  \cellcolor{red!25}4 & \cellcolor{green!60}7 & \cellcolor{green!60}7 & \cellcolor{green!60}7 & \cellcolor{blue!25} 1 & \cellcolor{blue!25} 1  & \cellcolor{blue!25} 1  & 0 & 0 & 0 & 7' & 7' & 7'  \\
\cellcolor{yellow}8 &  \cellcolor{green!25} 2 & \cellcolor{green!25} 2 & \cellcolor{green!25} 2 & \cellcolor{purple!60} 5 & \cellcolor{purple!60} 5 & \cellcolor{purple!60} 5 & \cellcolor{yellow}8 & \cellcolor{yellow}8 & 8' &  \cellcolor{green!25} 2 & \cellcolor{green!25} 2 & \cellcolor{green!25} 2 & 0 & 0 & 0 & 8' & 8'  \\
\cellcolor{red} 9 & \cellcolor{red} 9 & \cellcolor{orange!40}3 & \cellcolor{orange!40}3 & \cellcolor{orange!40}3  & \cellcolor{cyan!60} 6 & \cellcolor{cyan!60} 6 & \cellcolor{cyan!60} 6 & \cellcolor{red} 9 & 9' & 9' & 3' & 3' & 3' & 0 & 0 & 0 & 9'  \\
 \end{array}
 \right)
 \end{array}
$$
\caption{An illustration for the proof of Theorem~\ref{thm:upper_M_t,r}.
Consider the matrix that has two diagonal blocks that are equal to $D_{9,3}$ and ones elsewhere.
The matrix $D_{9,3}$ belongs to the family $\calM_{9,2}$. This follows from the choice of sets $A_l = \{l, l+6\}$ and $B_l = \{l, l+1, l+2\}$ for $l \in [9]$, where the elements are considered modulo $9$ in $[9]$ (see Lemma~\ref{lemma:DinM}).
In the partition of the ones described in the figure, the rectangles $A_1 \times B_1$ of the two diagonal blocks are merged into a single rectangle (labelled $1$), and the rectangles $A_2 \times B_2$ of the two diagonal blocks are merged into a single rectangle (labelled $2$). The extensions of the other rectangles $A_l \times B_l$ for $l \in [9] \setminus [2]$, as defined in the proof of Theorem~\ref{thm:upper_M_t,r}, are labelled in the figure by $l$ and $l'$ in the left and right blocks respectively.
It follows that the binary rank of the matrix is at most $16$.}
\label{fig:upperbound}	
\end{figure}

\subsection{The Complement of Circulant Block Diagonal Matrices}

The following lemma relates the matrices $D_{n,k}$ to the families $\calM_{t,r}$ (see Definition~\ref{def:M_t,r} and Figure~\ref{fig:upperbound}).

\begin{lemma}\label{lemma:DinM}
Let $n \geq 1$ be an integer.
\begin{enumerate}
  \item The matrices $D_{n,0}$ and $D_{n,n}$ belong to $\calM_{1,0}$.
  \item For every integer $k$ such that $n>k>0$, the matrix $D_{n,k}$ belongs to $\calM_{n,d-1}$ for $d = \gcd(n,k)$.
\end{enumerate}
\end{lemma}

\begin{proof}
It is easy to see that $D_{n,0} \in \calM_{1,0}$, as follows from the choice of the sets $A_1 = B_1 = [n]$ which satisfy the conditions of Definition~\ref{def:M_t,r} (in a somewhat trivial manner). Similarly, $D_{n,n} \in \calM_{1,0}$, as follows from the choice of the sets $A_1 = [n]$ and $B_1 = \emptyset$.

For the second item of the lemma, let $n$ and $k$ be integers satisfying $n>k>0$, and put $d= \gcd(n,k)$.
We turn to show that the matrix $D_{n,k}$ belongs to $\calM_{n,d-1}$.
To do so, consider the $n$ subsets of $[n]$ that include $d$ cyclically consecutive elements, that is, the sets
\[B_l = \{l,l+1,\ldots,l+d-1\}\]
for $l \in [n]$, where the elements are considered modulo $n$ with their representatives in $[n]$.
Observe that for every $i \in [n]$, the support of row $i$ in $D_{n,k}$ can be represented as a disjoint union of $(n-k)/d$ sets $B_l$ with $l = i~(\mod~d)$.
For every $l \in [n]$, let $A_l$ denote the set of rows $i \in [n]$ of $D_{n,k}$ for which this representation of row $i$ includes the set $B_l$.

We claim that the sets $A_l$ and $B_l$ for $l \in [n]$ satisfy the conditions of Definition~\ref{def:M_t,r}.
For Item~\ref{itm:M1}, it follows from the definition of the sets that the rectangles $A_l \times B_l$ for $l \in [n]$ form a partition of the ones in $D_{n,k}$.
For Item~\ref{itm:M2}, the definition of the sets implies that for every $l \in [n]$, the set $A_l$ includes only elements that are congruent to $l$ modulo $d$, so in particular, the sets $A_1, \ldots, A_{d-1}$ are pairwise disjoint.
For Item~\ref{itm:M3}, observe that the sets $B_l$ with $l = d~(\mod~ d)$, whose indices $l$ clearly satisfy $l \in [n] \setminus [d-1]$, form a partition of $[n]$.
Finally, for Item~\ref{itm:M4}, fix some $s \in [d-1]$, and observe that the sets $B_l$ with $l = s~(\mod~ d)$ and $l \neq s$, whose indices $l$ again satisfy $l \in [n] \setminus [d-1]$, form a partition of $[n] \setminus B_s$.
It therefore follows that $D_{n,k} \in \calM_{n,d-1}$, and we are done.
\end{proof}

Equipped with Theorem~\ref{thm:upper_M_t,r} and Lemma~\ref{lemma:DinM}, we restate and prove Theorem~\ref{thm:upper_new_ki}.

\newtheorem*{thma}{Theorem \ref{thm:upper_new_ki}}
\begin{thma}
Let $M$ be the complement of a $(k_1, \ldots, k_m; n_1, \ldots, n_m)$ circulant block diagonal matrix for integers satisfying $n_i \geq k_i > 0$ for all $i \in [m]$.
For every $i \in [m]$, put $d_i = \gcd(n_i,k_i)$ if $n_i >k_i$, and $d_i = 1$ if $n_i=k_i$.
Then,
\[\Rbin(M) \leq \Rreal(M)+\max_{i \in [m]} d_i-1.\]
\end{thma}

\begin{proof}
It may be assumed that the matrix $M$ is nonzero, as otherwise the statement of the theorem trivially holds.
Using the assumption $n_i \geq k_i > 0$ for all $i \in [m]$, we can apply Lemma~\ref{lemma:rank_M_comp} to the complement of $M$ to obtain that
\begin{eqnarray}\label{eq:rank_vs_comp}
\Rreal(\overline{M}) = \Rreal(M) = \sum_{i=1}^{m}{(n_i - \gcd(n_i,k_i)+1)}.
\end{eqnarray}

Put $I = \{ i \in [m] \mid n_i = k_i\}$.
By definition, we have $d_i=1$ for each $i \in I$, and $d_i = \gcd(n_i,k_i)$ for each $i \in [m] \setminus I$.
It follows from~\eqref{eq:rank_vs_comp} that
\[ \Rreal(M) = \sum_{i \in I}{(n_i - \gcd(n_i,k_i)+1)} + \sum_{i \in [m] \setminus I}{(n_i - \gcd(n_i,k_i)+1)} = |I| + \sum_{i \in [m] \setminus I}{(n_i-d_i+1)}.\]
By Lemma~\ref{lemma:DinM}, the $i$th diagonal block of $M$, which is the complement of $D_{n_i,n_i-k_i}$, belongs to $\calM_{1,0}$ for $i \in I$ and to $\calM_{n_i,d_i-1}$ for $i \in [m] \setminus I$.
By Theorem~\ref{thm:upper_M_t,r}, this implies that
\[\Rbin(M) \leq  |I| + \sum_{i \in [m] \setminus I}{(n_i-d_i+1)} + \max_{i \in [m]}{(d_i-1)} = \Rreal(M)+\max_{i \in [m]} d_i-1,\]
and the result follows.
\end{proof}

We next prove Theorem~\ref{thm:gap_kIntro}.

\begin{proof}[ of Theorem~\ref{thm:gap_kIntro}]
Write $r = 2k \cdot \ell +t$ for integers $\ell \geq 1$ and $0 \leq t \leq 2k-1$, and define $n = 2k \cdot \ell + k \cdot t$.
Let $M$ be the $n \times n$ block diagonal matrix that consists of $\ell$ diagonal blocks that are equal to $D_{2k,k}$ and $t$ diagonal blocks that are equal to $D_{k,0}$. Equivalently, $M$ is the $(k; 2k,\ldots,2k,k,\ldots,k)$ circulant block diagonal matrix, where $2k$ and $k$ appear, respectively, $\ell$ and $t$ times.
By Lemma~\ref{lemma:DinM}, it holds that $\overline{D_{2k,k}} \in \calM_{2k,k-1}$ and $\overline{D_{k,0}} \in \calM_{1,0}$.
Applying Theorem~\ref{thm:upper_M_t,r} to $\overline{M}$, we obtain that
\begin{eqnarray*}
\Rbin(\overline{M}) &\leq& (k+1) \cdot \ell+t +(k-1) \\
&=& (k+1) \cdot \ell+ \Big \lceil \frac{k+1}{2k} \cdot t\Big \rceil + t-\Big \lceil \frac{k+1}{2k} \cdot t\Big \rceil + (k-1) \\
& = & \Big \lceil \frac{k+1}{2k} \cdot r\Big \rceil + \Big \lfloor \frac{k-1}{2k} \cdot t\Big \rfloor + (k-1) \\
&\leq& \Big \lceil \frac{k+1}{2k} \cdot r\Big \rceil + (2k-3),
\end{eqnarray*}
where the last inequality holds using $t \leq 2k-1$.
This completes the proof.
\end{proof}

We end this section with the proof of Theorem~\ref{thm:2-regIntro}, which determines, almost exactly, the binary rank of $2$-regular matrices and of their complement.

\begin{proof}[ of Theorem~\ref{thm:2-regIntro}]
Let $M$ be a $2$-regular $0,1$ matrix of dimensions $n \times n$.
By Lemma~\ref{lemma:2-reg-structure}, $M$ is equal, up to permutations of rows and columns, to a $(2;n_1, \ldots,n_m)$ circulant block diagonal matrix for some integers $m \geq 1$ and $n_1,\ldots,n_m \geq 2$. Note that $n = \sum_{i=1}^{m}{n_i}$. Let $m_2$ denote the number of indices $i \in [m]$ with $n_i=2$, and let $m_{even}$ denote the number of indices $i \in [m]$ for which $n_i$ is even.

Suppose that $M$ is not the all-one $2 \times 2$ matrix.
By Lemma~\ref{lemma:2-reg-rankR}, it holds that
\begin{eqnarray}\label{eq:rank_M2-reg}
\Rreal(M) = \Rreal(\overline{M}) = n - m_{even}.
\end{eqnarray}
Since $M$ is a block diagonal matrix, its binary rank is equal to the sum of the binary rank of its diagonal blocks.
It clearly holds that $\Rbin(D_{2,0}) = 1$, and Lemma~\ref{lemma:isolation} implies that for every $i \in [m]$ with $n_i >2$, it holds that $\Rbin(D_{n_i,n_i-2}) = n_i$.
This implies that
\[ \Rbin(M) = \sum_{i=1}^{m}{\Rbin(D_{n_i,n_i-2})} = n-m_2.\]

As for the complement of $M$, the real rank forms a lower bound on its binary rank. Hence, it follows from~\eqref{eq:rank_M2-reg} that
\[\Rbin(\overline{M}) \geq \Rreal(\overline{M}) = n - m_{even}.\]
On the other hand, Theorem~\ref{thm:upper_new_ki} combined with~\eqref{eq:rank_M2-reg} implies that
\[ \Rbin(\overline{M}) \leq \Rreal(\overline{M}) + 1 = n-m_{even}+1.\]
It thus follows that $\Rbin(\overline{M}) \in \{n-m_{even}, n-m_{even}+1\}$, as required.
\end{proof}

\section{Lower Bounds}\label{sec:lower}

In this section, we prove for several families of circulant block diagonal matrices that the binary rank of their complement is strictly larger than the real rank.
We start with some lemmas that will be used later to prove the lower bounds on the binary rank.

\subsection{Lemmas}

Consider the following definition.
\begin{definition}[Row and Column Sequences]
Let $M$ be the complement of a $(k_1, \ldots, k_m ; n_1, \ldots, n_m)$ circulant block diagonal matrix for integers satisfying $n_i \geq k_i \geq 0$, and let $R = A \times B$ be a rectangle of ones in $M$.
For any $i \in [m]$, put $d_i = \gcd(n_i,k_i)$.
The {\em row sequence of $R$ in block $i$} is defined as the sequence of integers $(a_1, \ldots, a_{d_i})$, where for each $t \in [d_i]$, $a_t$ is the number of rows of $R$ in the $i$th diagonal block of $M$ whose indices in the block are congruent to $t$ modulo $d_i$, i.e., the number of rows $(i,j) \in A$ of $R$ with $j = t~(\mod~d_i)$.
If $a_1 = a_2 =\cdots = a_{d_i}$, then we say that $R$ is {\em balanced in rows in block $i$}.
Similarly, the {\em column sequence of $R$ in block $i$} is defined as the sequence of integers $(b_1, \ldots, b_{d_i})$, where for each $t \in [d_i]$, $b_t$ is the number of columns of $R$ in the $i$th diagonal block of $M$ whose indices in the block are congruent to $t$ modulo $d_i$, i.e., the number of columns $(i,j) \in B$ of $R$ with $j = t~(\mod~d_i)$.
If $b_1 = b_2 = \cdots = b_{d_i}$, then we say that $R$ is {\em balanced in columns in block $i$}.
\end{definition}

The following lemma relates the size of a partition of the ones in the complement of a circulant block diagonal matrix to the row and column sequences of its rectangles.
We will use it to show that if the partition includes a rectangle that is not balanced in rows or in columns in some block then its size is larger than the real rank of the matrix (see Corollary~\ref{cor:balanced+1}).

\begin{lemma}\label{lemma:lin_ind+1}
Let $M$ be the complement of a $(k_1, \ldots, k_m ; n_1, \ldots, n_m)$ circulant block diagonal matrix for integers satisfying $n_i \geq k_i \geq 0$ for all $i \in [m]$, and let $P$ be a partition of the ones in $M$ into combinatorial rectangles.
For some $h \in [m]$, put $d_{h} = \gcd(n_{h},k_{h})$.
Suppose that $P$ includes $\ell$ rectangles whose row (or column) sequences in block $h$ are $a^{(1)}, \ldots, a^{(\ell)} \in \Z^{d_{h}}$, and let $a^{(0)}$ denote the all-one vector in $\Z^{d_{h}}$. If the vectors of $\{ a^{(0)}, a^{(1)}, \ldots, a^{(\ell)}\}$ are linearly independent, then
\[|P| \geq \Rreal(M) + \ell.\]
\end{lemma}

\begin{proof}
Let $M$ and $P$ be a matrix and a partition of its ones into combinatorial rectangles as in the statement of the lemma.
Let $R_1, \ldots, R_\ell \in P$ be the $\ell$ rectangles, viewed as $0,1$ matrices, whose row sequences in block $h$ are respectively $a^{(1)}, \ldots, a^{(\ell)} \in \Z^{d_{h}}$ where $d_{h}=\gcd(n_{h},k_{h})$, and suppose that they form together with $a^{(0)}$ a linearly independent set.

Consider the $0,1$ matrix $M' = M -\sum_{s=1}^{\ell}{R_s}$.
We turn to prove that $\Rreal(M') \geq \Rreal(M)$. Since the rectangles of $P$ excluding $R_1, \ldots, R_\ell$ form a partition of the ones in $M'$ into combinatorial rectangles, this will imply that
\[|P| \geq \Rbin(M') + \ell \geq \Rreal(M') + \ell \geq \Rreal(M) + \ell,\]
and will thus complete the proof.

To prove that $\Rreal(M') \geq \Rreal(M)$, consider the system of linear equations
\begin{equation}\label{eq:M'}
M' \cdot x = 0,
\end{equation}
where $x$ is a real vector whose entries are naturally denoted by $x_{i,j}$ for $i \in [m]$ and $j \in [n_i]$, such that $x_{i,j}$ is the $j$th entry of the $i$th block in $x$.
Throughout the proof, the indices $j$ in the pairs $(i,j)$ are considered modulo $n_i$ with their representatives in $[n_i]$.
For every $i \in [m]$, put $X_i = \sum_{j \in [n_i]}{x_{i,j}}$, and let $X = \sum_{i =1}^{m}{X_i}$ denote the sum of all the entries of $x$.
Further, for every $s \in [\ell]$, put $R_s = A_s \times B_s$, and let $y_s = \sum_{(i,j) \in B_s}{x_{i,j}}$ denote the sum of the variables that correspond to the columns of $R_s$.
Notice that for each $i \in [m]$, the first row of the $i$th diagonal block of $M$ starts with $k_i$ zeros followed by $n_i-k_i$ ones.
This implies that the $(i,j)$ equation of~\eqref{eq:M'} can be written as
\[X - (x_{i,j}+x_{i,j+1}+ \cdots + x_{i,j+k_i-1}) = z_{i,j},\]
where $z_{i,j}$ is the sum of the variables that correspond to the columns of the rectangles among $R_1, \ldots, R_\ell$ that include the row $(i,j)$.
Equivalently, $z_{i,j}$ is the sum of $y_s$ over all $s \in [\ell]$ satisfying $(i,j) \in A_s$.

For every $t \in [d_{h}]$, consider the $n_{h}/d_{h}$ equations of~\eqref{eq:M'} indexed by $(h,j)$ with $j = t~(\mod~d_{h})$. Observe that every variable $x_{i,j}$ with $i \neq h$ appears with coefficient $1$ in the left hand side of each of these equations.
For $i = h$, however, the variable $x_{i,j}$ has coefficient $0$ in precisely $k_{h}/d_{h}$ of them and $1$ elsewhere.
Letting $\widetilde{X}$ denote the sum of the left hand side of these $n_{h}/d_{h}$ equations, it follows that
\[\widetilde{X} = \tfrac{n_{h}}{d_{h}} \cdot X- \tfrac{k_{h}}{d_{h}} \cdot X_{h}.\]
Note that $\widetilde{X}$ is independent of $t$.
As for the right hand side, the sum of these equations includes the variables $y_s$ for $s \in [\ell]$, where the coefficient of $y_s$ is the number of rows $(h,j)$ with $j = t~(\mod~d_{h})$ that belong to the set of rows $A_s$ of the rectangle $R_s$. By definition, this coefficient is the $t$th entry of the row sequence $a^{(s)}$ of the rectangle $R_s$.
We thus get, for each $t \in [d_{h}]$, the equation
\[\widetilde{X} = \sum_{s \in [\ell]}{a^{(s)}_t \cdot y_s}.\]
This can be viewed as a system of $d_{h}$ linear equations in the variables $\widetilde{X}$ and $y_s$ for $s \in [\ell]$.
Since the vectors of $\{ a^{(0)}, a^{(1)}, \ldots, a^{(\ell)}\}$ are linearly independent, it follows that $y_s = 0$ for all $s \in [\ell]$, and therefore $z_{i,j}=0$ for all pairs $(i,j)$. This implies that any solution $x$ of the system~\eqref{eq:M'} is also a solution of the system $M \cdot x = 0$. Hence, $\Rreal(M') \geq \Rreal(M)$, and we are done.
\end{proof}

\begin{corollary}\label{cor:balanced+1}
Let $M$ be the complement of a $(k_1, \ldots, k_m ; n_1, \ldots, n_m)$ circulant block diagonal matrix for integers satisfying $n_i \geq k_i \geq 0$ for all $i \in [m]$, and let $P$ be a partition of the ones in $M$ into combinatorial rectangles.
Suppose that $P$ includes a rectangle that is not balanced in rows (or columns) in some block.
Then, $|P| \geq \Rreal(M) + 1$.
\end{corollary}

\begin{proof}
Suppose that $P$ includes a rectangle that is not balanced in rows (or columns) in some block $h \in [m]$, and put $d_h = \gcd(n_h,k_h)$.
This implies that the row (or column) sequence of this rectangle in block $h$ and the all-one vector in $\Z^{d_h}$ are linearly independent.
By applying Lemma~\ref{lemma:lin_ind+1} with $\ell=1$, the result follows.
\end{proof}

The following lemma relates the size of a partition of the ones in the complement of a circulant block diagonal matrix to the numbers of rows and columns of its rectangles in each block.

\begin{lemma}\label{lemma:divides+1}
Let $M$ be the complement of a $(k_1, \ldots, k_m ; n_1, \ldots, n_m)$ circulant block diagonal matrix for integers satisfying $n_i \geq k_i >0$ for all $i \in [m]$, and let $P$ be a partition of the ones in $M$ into combinatorial rectangles.
For each $i \in [m]$, put $d_i = \gcd(n_i,k_i)$.
Let $R = A \times B$ be a rectangle of $P$, and for each $i \in [m]$, put $A_i = A \cap (\{i\} \times [n_i])$ and $B_i = B \cap (\{i\} \times [n_i])$.
Define
\[S = \Big ( \sum_{i=1}^{m}{\tfrac{|A_i|}{k_i}} \Big ) \cdot \Big ( \sum_{i=1}^{m}{\tfrac{|B_i|}{k_i}} \Big )~~~~~\mbox{and}~~~~~N = \sum_{i=1}^{m}{\tfrac{n_i}{k_i}} - 1.\]
If $S$ cannot be written as $S = e \cdot N$ for a number $e$ of the form $e = \sum_{i=1}^{m}{e_i \cdot \frac{d_i}{k_i}}$ for integers $e_1, \ldots, e_m \in \Z$, then, $|P| \geq \Rreal(M) + 1$.
\end{lemma}

\begin{proof}
If the given rectangle $R = A \times B$ is not balanced in rows or in columns in some block, then it follows from Corollary~\ref{cor:balanced+1} that $|P| \geq \Rreal(M) + 1$ and we are done. It thus can be assumed that $R$ is balanced in rows and in columns in all blocks.

Viewing $R$ as a $0,1$ matrix, consider the $0,1$ matrix $M' = M -R$.
We turn to prove that $\Rreal(M') \geq \Rreal(M)$. Since the rectangles of $P$ excluding $R$ form a partition of the ones in $M'$ into combinatorial rectangles, this will imply that
\[|P| \geq \Rbin(M') + 1 \geq \Rreal(M') + 1 \geq \Rreal(M) + 1,\]
and will thus complete the proof.

To prove that $\Rreal(M') \geq \Rreal(M)$, consider the system of linear equations
\begin{equation}\label{eq:M'2}
M' \cdot x = 0,
\end{equation}
where $x$ is a real vector whose entries are denoted by $x_{i,j}$ for $i \in [m]$ and $j \in [n_i]$, such that $x_{i,j}$ is the $j$th entry of the $i$th block in $x$.
As before, the indices $j$ in the pairs $(i,j)$ are considered modulo $n_i$ with their representatives in $[n_i]$.
For every $i \in [m]$, put $X_i = \sum_{j \in [n_i]}{x_{i,j}}$, and let $X = \sum_{i =1}^{m}{X_i}$.
Further, let $y = \sum_{(i,j) \in B}{x_{i,j}}$ denote the sum of the variables that correspond to the columns of $R$.
Observe that the $(i,j)$ equation of~\eqref{eq:M'2} can be written as
\begin{eqnarray}\label{eq:equation_M'}
X - (x_{i,j}+x_{i,j+1}+ \cdots + x_{i,j+k_i-1}) = z_{i,j},
\end{eqnarray}
where $z_{i,j} = y$ if $(i,j) \in A$ and $z_{i,j} = 0$ otherwise.

For each $i \in [m]$, consider the $n_i/d_i$ equations of the $i$th block whose indices in the block are congruent modulo $d_i$ to, say, $1$, i.e., the equations indexed by $(i,j)$ with $j = 1~(\mod ~d_i)$.
Observe that every variable $x_{i',j'}$ with $i' \neq i$ appears with coefficient $1$ in the left hand side of each of these equations.
For $i' = i$, however, the variable $x_{i',j'}$ has coefficient $0$ in precisely $k_{i}/d_{i}$ of the equations and $1$ elsewhere.
This implies that the sum of the left hand side of these $n_{i}/d_{i}$ equations is
\[\tfrac{n_{i}}{d_{i}} \cdot X- \tfrac{k_{i}}{d_{i}} \cdot X_{i}.\]
As for the right hand side, the sum of these equations is the variable $y$ multiplied by the number of rows $(i,j')$ with $j' = 1~(\mod~d_{i})$ that belong to the rectangle $R$. Since $R$ is balanced in rows in block $i$, this coefficient is $|A_i|/d_i$.
Hence, for each $i \in [m]$, we derive the equation
\[\tfrac{n_i}{d_i} \cdot X - \tfrac{k_i}{d_i} \cdot X_i = \tfrac{|A_i|}{d_i} \cdot y,\]
and by multiplying it by $d_i/k_i$, we obtain the equation
\[\tfrac{n_i}{k_i} \cdot X -  X_i = \tfrac{|A_i|}{k_i} \cdot y.\]
The sum of these $m$ equations gives the equation
\begin{eqnarray}\label{eq:EquationX-y-a}
\Big ( \sum_{i=1}^{m}{\tfrac{n_i}{k_i}} -1 \Big ) \cdot X = \Big (\sum_{i=1}^{m}{\tfrac{|A_i|}{k_i}} \Big ) \cdot y.
\end{eqnarray}

We turn to obtain another equation relating $X$ and $y$.
To do so, let us first consider, for any block $i \in [m]$ with $n_i > k_i$, the $n_i$ equations obtained by subtracting equation $(i,j)$, given in~\eqref{eq:equation_M'}, from equation $(i,j+1)$ for $j \in [n_i]$ (again, the indices $j$ are considered modulo $n_i$ in $[n_i]$). Every such equation has the form
\[x_{i,j} - x_{i,j+k_i} = z_{i,j+1} - z_{i,j}.\]
Recalling that each $z_{i,j}$ is either $y$ or $0$, this implies that every variable $x_{i,j}$ can be written as the sum of the variable $x_{i,j+k_i}$, which is the cyclic shift of $x_{i,j}$ by $k_i$ locations in the block, and an integer multiple of $y$.
Moreover, the fact that $d_i = \gcd(n_i,k_i)$ implies that $d_i = a \cdot k_i ~(\mod ~n_i)$ for some integer $a \geq 0$. Hence, by applying the above argument $a$ times, it follows that $x_{i,j}$ can be written as the sum of the variable $x_{i,j+d_i}$ and an integer multiple of $y$.
By repeatedly applying this argument, it follows that for every $j \in [n_i]$ there exists some $c_{i,j} \in \Z$ such that
\begin{eqnarray}\label{eq:x_{i,j}}
x_{i,j} = x_{i,t} + c_{i,j} \cdot y,
\end{eqnarray}
for the $t \in [d_i]$ that satisfies $j = t~(\mod~ d_i)$.
By substituting the equations of~\eqref{eq:x_{i,j}} into the first equation of block $i$, indexed $(i,1)$, it follows that there exists an integer $e_i \in \Z$ for which
\[X - \tfrac{k_i}{d_i} \cdot \sum_{t=1}^{d_i}{x_{i,t}} = e_i \cdot y,\]
and by rearranging, we obtain that
\begin{eqnarray}\label{eq:X-x_{i,j}}
\tfrac{1}{d_i} \cdot \sum_{t=1}^{d_i}{x_{i,t}} = \tfrac{1}{k_i} \cdot (X-e_i \cdot y).
\end{eqnarray}
Note that for an $i \in [m]$ that satisfies $n_i=k_i$, such an equality is directly given by any equation of the $i$th block in~\eqref{eq:M'2}.

By substituting the equations of~\eqref{eq:x_{i,j}} into the definition of $y = \sum_{(i,j) \in B}{x_{i,j}}$, using the fact that the rectangle $R$ is balanced in columns in all blocks, we derive for some $\widetilde{c} \in \Z$ the equation
\[ y = \tfrac{|B_1|}{d_1} \cdot \sum_{t=1}^{d_1}{x_{1,t}} + \tfrac{|B_2|}{d_2} \cdot \sum_{t=1}^{d_2}{x_{2,t}} + \cdots + \tfrac{|B_m|}{d_m} \cdot \sum_{t=1}^{d_m}{x_{m,t}} + \widetilde{c} \cdot y.\]
By combining this equation with the equations of~\eqref{eq:X-x_{i,j}}, we obtain the equation
\[ y = \tfrac{|B_1|}{k_1} \cdot (X - e_1 \cdot y) + \tfrac{|B_2|}{k_2} \cdot (X - e_2 \cdot y) + \cdots + \tfrac{|B_m|}{k_m} \cdot (X - e_m \cdot y) + \widetilde{c} \cdot y,\]
which by rearranging is equivalent to
\begin{eqnarray}\label{eq:EquationX-y-b}
\Big ( \sum_{i=1}^{m}{\tfrac{|B_i|}{k_i}} \Big ) \cdot X = \Big ( 1-\widetilde{c}+ \sum_{i=1}^{m}{\tfrac{e_i \cdot |B_i|}{k_i}} \Big ) \cdot y.
\end{eqnarray}

Finally, since $R$ is balanced in columns in all blocks, it follows that $d_i$ divides $|B_i|$ for each $i \in [m]$.
By combining this with the assumption of the lemma, it follows that
\begin{eqnarray}\label{eq:condition_X_y}
\Big ( \sum_{i=1}^{m}{\tfrac{|A_i|}{k_i}} \Big ) \cdot \Big ( \sum_{i=1}^{m}{\tfrac{|B_i|}{k_i}} \Big )  \neq \Big ( 1-\widetilde{c}+ \sum_{i=1}^{m}{e_i \cdot \tfrac{|B_i|}{k_i}} \Big ) \cdot \Big ( \sum_{i=1}^{m}{\tfrac{n_i}{k_i}} -1 \Big ).
\end{eqnarray}
Consider the equations given in~\eqref{eq:EquationX-y-a} and~\eqref{eq:EquationX-y-b} in the variables $X$ and $y$, and observe that using~\eqref{eq:condition_X_y} they imply that $X = y =0$.
Therefore, any solution $x$ of the system~\eqref{eq:M'2} further satisfies $M \cdot x = 0$, hence $\Rreal(M') \geq \Rreal(M)$, and we are done.
\end{proof}

\begin{remark}\label{remark:lemma_common_k}
Let us mention here a couple of useful special cases of Lemma~\ref{lemma:divides+1}.
\begin{enumerate}
  \item\label{itm:remark_1} For the case where $d_i = k_i$ for all $i \in [m]$, the condition on the numbers $S$ and $N$ in the lemma is equivalent to the condition that $S$ is an integer multiple of $N$.
  \item\label{itm:remark_2} For the case where $k_1, \ldots,k_m$ are all equal to some integer $k$, it holds that
  \[S = \tfrac{1}{k^2} \cdot \sum_{i=1}^{m}{|A_i|} \cdot \sum_{i=1}^{m}{|B_i|} = \tfrac{1}{k^2} \cdot |A| \cdot |B|~~~~~\mbox{and}~~~~~N = \sum_{i=1}^{m}{\tfrac{n_i}{k}} - 1 = \tfrac{1}{k} \cdot (n-k),\]
  where $n = \sum_{i=1}^{m}{n_i}$. Hence, the condition on the numbers $S$ and $N$ in the lemma is equivalent to the condition that $|A| \cdot |B|$ cannot be written as $|A| \cdot |B| = e \cdot (n-k)$ where $e = \sum_{i =1}^{m}{e_i \cdot d_i}$ for some $e_1, \ldots, e_m \in \Z$.
\end{enumerate}
\end{remark}

\subsection{The Binary Rank versus the Real Rank}

We turn to present several families of circulant block diagonal matrices for which the complement satisfies that its binary rank is strictly larger than its real rank.
We start with the following simple consequence of Corollary~\ref{cor:balanced+1}.

\begin{theorem}\label{thm:ni=ki+di}
Let $M$ be the complement of a $(k_1, \ldots, k_m ; n_1, \ldots, n_m)$ circulant block diagonal matrix for integers satisfying $n_i \geq k_i \geq 0$ for all $i \in [m]$.
Suppose that some $j \in [m]$ satisfies $n_j=k_j+\gcd(n_j,k_j)$ and $\gcd(n_j,k_j) > 1$.
Then, $\Rbin(M) \geq \Rreal(M)+1$.
\end{theorem}

\begin{proof}
Let $P$ be a partition of the ones in $M$ into combinatorial rectangles.
Consider the $j$th diagonal block of $M$, which is equal to the complement of $D_{n_j,n_j-k_j}$, and notice that each of its rows consists of $n_j-k_j$ ones.
By $n_j > k_j$, this diagonal block is nonzero, hence there exists a rectangle $R \in P$ that includes a row and a column of the $j$th block of $M$.
We turn to show that this $R$ is not balanced in rows or in columns in block $j$.
This will imply, by Corollary~\ref{cor:balanced+1}, that $|P| \geq \Rreal(M) + 1$, and will thus complete the proof.

Assume for the sake of contradiction that $R$ is balanced in rows and in columns in block $j$, and put $d = \gcd(n_j,k_j) = n_j-k_j$.
It follows that $R$ includes at least $d$ rows and at least $d$ columns in block $j$, and therefore, it must cover all the $d$ ones in at least $d> 1$ rows of the $j$th diagonal block of $M$.
But this is impossible because the rows of the $j$th diagonal block of $M$ are pairwise distinct, and therefore a rectangle cannot cover all the ones in two of them.
\end{proof}

We next prove Theorem~\ref{thm:ki_divides_ni+1}, whose proof uses the following lemma.

\begin{lemma}\label{lemma:weighted_rec}
Let $M$ be the complement of a $(k_1, \ldots, k_m ; n_1, \ldots, n_m)$ circulant block diagonal matrix for integers satisfying $n_i \geq k_i >0$ for all $i \in [m]$.
Suppose that for each $i \in [m]$, $k_i$ divides $n_i$, and that some $i \in [m]$ satisfies $n_i > k_i >1$.
Every partition $P$ of the ones in $M$ into combinatorial rectangles either satisfies $|P| \geq \Rreal(M) + 1$, or includes a rectangle $R = A \times B$ such that
\[ \Big ( \sum_{i=1}^{m}{\tfrac{|A_i|}{k_i}} \Big ) \cdot \Big ( \sum_{i=1}^{m}{\tfrac{|B_i|}{k_i}} \Big ) < \sum_{i=1}^{m}{\tfrac{n_i}{k_i}} - 1,\]
where $A_i = A \cap (\{i\} \times [n_i])$ and $B_i = B \cap (\{i\} \times [n_i])$ for all $i \in [m]$.
\end{lemma}

\begin{proof}
The assumptions of the lemma imply that $k_i = \gcd(n_i,k_i)$ for all $i \in [m]$.
Since some $i \in [m]$ satisfies $n_i > k_i > 0$, it follows that $M$ is nonzero.
By applying Lemma~\ref{lemma:rank_M_comp} to $\overline{M}$, we obtain that $\Rreal(M) = \sum_{i=1}^{m}{(n_i-k_i+1)}$.
If the given partition satisfies $|P| \geq \Rreal(M)+1$, then we are done.
So suppose that it satisfies $|P| \leq \Rreal(M)$. Since the real rank forms a lower bound on the binary rank, it follows that
\begin{eqnarray}\label{eq:|P|}
|P| = \Rreal(M) = \sum_{i=1}^{m}{(n_i-k_i+1)}.
\end{eqnarray}
For every rectangle $R = A \times B$ of $P$, consider the quantity
\[ g(R) = \Big ( \sum_{i=1}^{m}{\tfrac{|A_i|}{k_i}} \Big ) \cdot \Big ( \sum_{i=1}^{m}{\tfrac{|B_i|}{k_i}} \Big ) = \sum_{i,j \in [m]}{\tfrac{|A_i| \cdot |B_j|}{k_i \cdot k_j}}.\]
Observe that for every $i,j \in [m]$, the sum of the term $|A_i| \cdot |B_j|$ over all rectangles $R = A \times B$ of $P$ precisely counts the ones of the matrix $M$ in the $(i,j)$ block. The latter is equal to $n_i \cdot n_j$ whenever $i \neq j$ and to $n_i \cdot (n_i - k_i)$ whenever $i=j$. It thus follows that
\begin{eqnarray*}
\sum_{R \in P}{g(R)} &=& \sum_{i \neq j \in [m]}{\tfrac{n_i \cdot n_j}{k_i \cdot k_j}} + \sum_{i =1}^{m}{\tfrac{n_i \cdot (n_i-k_i)}{k_i^2}} = \Big (\sum_{i =1}^{m}{\tfrac{n_i}{k_i}} \Big )^2 - \sum_{i =1}^{m}{\tfrac{n_i}{k_i}}\\ &=& \Big (\sum_{i=1}^{m}{\tfrac{n_i}{k_i}} \Big ) \cdot \Big (\sum_{i=1}^{m}{\tfrac{n_i}{k_i}} -1\Big ).
\end{eqnarray*}
It can be checked that for every $i \in [m]$, it holds that $\frac{n_i}{k_i} \leq n_i-k_i+1$, and that a strict inequality holds whenever $n_i > k_i >1$. By assumption, there exists an $i \in [m]$ with a strict inequality, which using~\eqref{eq:|P|} implies that
\[ \sum_{R \in P}{g(R)} < \sum_{i=1}^{m}{(n_i-k_i+1)} \cdot  \Big (\sum_{i =1}^{m}{\tfrac{n_i}{k_i}} -1\Big ) = |P| \cdot \Big (\sum_{i=1}^{m}{\tfrac{n_i}{k_i}} -1\Big ).\]
This implies that there exists a rectangle $R \in P$ satisfying $g(R) < \sum_{i=1}^{m}{\tfrac{n_i}{k_i}} -1$, as desired.
\end{proof}

We now restate and prove Theorem~\ref{thm:ki_divides_ni+1}.

\newtheorem*{thmb}{Theorem \ref{thm:ki_divides_ni+1}}
\begin{thmb}
Let $M$ be the complement of a $(k_1, \ldots, k_m; n_1, \ldots, n_m)$ circulant block diagonal matrix for integers satisfying $n_i \geq k_i > 0$ for all $i \in [m]$.
Suppose that for each $i \in [m]$, $k_i$ divides $n_i$, and that some $i \in [m]$ satisfies $n_i > k_i >1$.
Then, $\Rbin(M) \geq \Rreal(M) + 1$.
\end{thmb}

\begin{proof}
Let $P$ be a partition of the ones in $M$ into combinatorial rectangles.
By Lemma~\ref{lemma:weighted_rec}, we either have $|P| \geq \Rreal(M) + 1$, or there exists a rectangle $R = A \times B$ in $P$ satisfying
\begin{eqnarray}\label{eq:rectangle_di=ki}
\Big ( \sum_{i=1}^{m}{\tfrac{|A_i|}{k_i}} \Big ) \cdot \Big ( \sum_{i=1}^{m}{\tfrac{|B_i|}{k_i}} \Big ) < \sum_{i=1}^{m}{\tfrac{n_i}{k_i}} - 1,
\end{eqnarray}
where $A_i = A \cap (\{i\} \times [n_i])$ and $B_i = B \cap (\{i\} \times [n_i])$ for each $i \in [m]$.
For the latter case, apply Lemma~\ref{lemma:divides+1} with this rectangle $R$ and with $d_i = \gcd(n_i,k_i)=k_i$ for $i \in [m]$ (see Remark~\ref{remark:lemma_common_k}, Item~\ref{itm:remark_1}).
Since the left hand side of~\eqref{eq:rectangle_di=ki} is not an integer multiple of the right hand side, it follows that $|P| \geq \Rreal(M) + 1$, and we are done.
\end{proof}

Our next result is the following.

\begin{theorem}\label{thm:common_k_d}
Let $M$ be the complement of a $(k ; n_1, \ldots, n_m)$ circulant block diagonal matrix for integers satisfying $n_i \geq k >0$ for all $i \in [m]$.
Put $d_i = \gcd(n_i,k)$ for every $i \in [m]$, and define $n = \sum_{i=1}^{m}{n_i}$.
Suppose that there exists an integer $d >1$ that divides $d_i$ for all $i \in [m]$ and that satisfies $\Rreal(M) > \frac{n}{d}$.
Then, $\Rbin(M) \geq \Rreal(M) + 1$.
\end{theorem}

\begin{proof}
Let $P$ be a partition of the ones in $M$ into combinatorial rectangles.
The assumption implies that $|P| \geq \Rreal(M) > \frac{n}{d}$.
Since the total number of ones in $M$ is $n \cdot (n-k)$, it follows that there exists a rectangle $R = A \times B$ in $P$ for which it holds that
\[|A| \cdot |B| \leq \frac{n \cdot (n-k)}{|P|} < d \cdot (n-k).\]
Since $d$ divides $d_i$ for all $i \in [m]$, it follows that $|A| \cdot |B|$ cannot be written as $e \cdot (n-k)$ where $e = \sum_{i=1}^{m}{e_i \cdot d_i}$ and $e_1,\ldots,e_m \in \Z$.
Applying Lemma~\ref{lemma:divides+1} with the rectangle $R$ (see Remark~\ref{remark:lemma_common_k}, Item~\ref{itm:remark_2}), it follows that $|P| \geq \Rreal(M) + 1$, as required.
\end{proof}

As a corollary of Theorem~\ref{thm:common_k_d}, we obtain the following result.

\begin{theorem}\label{thm:common_k_d_same}
Let $M$ be the complement of a $(k ; n_1, \ldots, n_m)$ circulant block diagonal matrix for integers satisfying $n_i \geq k >0$ for all $i \in [m]$.
Suppose that there exists an integer $d >1$ such that $d = \gcd(n_i,k)$ for all $i \in [m]$ and such that $n_i > d$ for some $i \in [m]$.
Then, $\Rbin(M) \geq \Rreal(M) + 1$.
\end{theorem}

\begin{proof}
Put $n = \sum_{i=1}^{m}{n_i}$, and apply Lemma~\ref{lemma:rank_M_comp} to obtain that the real rank of $M$ satisfies
\[\Rreal(M) = \sum_{i=1}^{m}{(n_i-d+1)} = n - m \cdot (d-1) > \tfrac{n}{d},\]
where the inequality holds because $d>1$ and because $n > m \cdot d$, which follows from the fact that $n_i \geq d$ for all $i \in [m]$ with a strict inequality for some $i \in [m]$.
This allows us to apply Theorem~\ref{thm:common_k_d} and to complete the proof.
\end{proof}

As an immediate corollary, we obtain that the binary rank of $D_{n,k}$ exceeds its real rank, whenever the latter is not $n$, and thus confirm Theorem~\ref{thm:D_n_k}.

\begin{proof}[ of Theorem~\ref{thm:D_n_k}]
Let $n>k>0$ be integers.
If $\gcd(n,k)=1$, then it follows from Lemma~\ref{lemma:realrank_D} that $\Rbin(D_{n,k}) = \Rreal(D_{n,k})=n$.
Otherwise, it holds that $n>k \geq \gcd(n,k)>1$, allowing us to apply Theorem~\ref{thm:common_k_d_same} with $m=1$ and to obtain that
$\Rbin(D_{n,k}) \geq \Rreal(D_{n,k}) + 1$, so we are done.
\end{proof}

We end this section with the following result that deals with the complement of $k$-regular $n \times n$ circulant block diagonal matrices for which $n-k$ is a prime.

\begin{theorem}\label{thm:prime}
Let $M$ be the complement of a $(k ; n_1, \ldots, n_m)$ circulant block diagonal matrix for integers satisfying $n_i \geq k >0$ for all $i \in [m]$, and put $n = \sum_{i=1}^{m}{n_i}$.
Suppose that $n-k$ is a prime number.
Then, $\Rbin(M) \geq \min( \Rreal(M) + 1,n)$.
\end{theorem}

\begin{proof}
Let $P$ be a partition of the ones in $M$ into combinatorial rectangles.
If there exists a rectangle $R = A \times B$ in $P$ such that $|A| \cdot |B|$ is not an integer multiple of $n-k$, then by Lemma~\ref{lemma:divides+1} (see Remark~\ref{remark:lemma_common_k}, Item~\ref{itm:remark_2}), it holds that $|P| \geq \Rreal(M) + 1$. Otherwise, all the rectangles $R = A \times B$ of $P$ satisfy that $|A| \cdot |B|$ is an integer multiple of $n-k$. Since $n-k$ is a prime, using the fact that $M$ is an $(n-k)$-regular matrix with distinct rows, it follows that every such rectangle satisfies either $|A|=n-k,~|B|=1$ or $|A|=1,~|B|=n-k$, and thus $|A| \cdot |B| = n-k$. Since the total number of ones in $M$ is $n \cdot (n-k)$, it follows that $|P| \geq n$, so we are done.
\end{proof}

\section{Matrices with Regularity $2$}\label{sec:2-reg}

In this section, we further explore $2$-regular matrices.
The following two theorems determine the exact binary rank of families of matrices whose complement is $2$-regular.

\begin{theorem}\label{thm:2_even}
For even integers $n_1, \ldots, n_m \geq 2$, let $M$ be the complement of a $(2 ; n_1, \ldots, n_m)$ circulant block diagonal matrix. Put $n = \sum_{i=1}^{m}{n_i}$, and suppose that $n_i >2$ for some $i \in [m]$.
Then, \[\Rbin(M) = n-m+1.\]
\end{theorem}

\begin{proof}
For every $i \in [m]$, it holds that $\gcd(n_i,2)=2$. By Lemma~\ref{lemma:2-reg-rankR}, the real rank of $M$ satisfies $\Rreal(M) = n-m$.
By Theorem~\ref{thm:2-regIntro}, it holds that $\Rbin(M) \in \{n-m,n-m+1\}$, and by Theorem~\ref{thm:ki_divides_ni+1}, we obtain that $\Rbin(M) \geq \Rreal(M)+1 = n-m+1$, so we are done.
\end{proof}

\begin{theorem}\label{thm:2_even_odd}
For integers $n_1, \ldots, n_m \geq 2$, let $M$ be the complement of a $(2 ; n_1, \ldots, n_m)$ circulant block diagonal matrix.
Let $m_{even}$ denote the number of indices $i \in [m]$ for which $n_i$ is even.
Let $n_{even}$ denote the sum of the $n_i$'s that are even, and let $n_{odd}$ denote the sum of the $n_i$'s that are odd. Put $n = n_{even} + n_{odd}$.
Suppose that $n$ is odd and that
\[n > 2 \cdot m_{even} + n_{odd} \cdot (n_{odd}-2).\]
Then, $\Rbin(M) = n-m_{even}+1$.
\end{theorem}

\begin{proof}
By Theorem~\ref{thm:2-regIntro}, it holds that $\Rbin(M) \leq n-m_{even}+1$.
To prove a matching lower bound, let $P$ be a partition of the ones in $M$ into combinatorial rectangles.
Using Lemma~\ref{lemma:2-reg-rankR}, we clearly have
\[|P| \geq \Rreal(M) = n-m_{even}.\]
If the partition $P$ includes a rectangle that is not balanced in rows or in columns in some block, then Corollary~\ref{cor:balanced+1} implies that $|P| \geq n-m_{even}+1$, and we are done.
If the partition $P$ includes a rectangle $R = A \times B$ such that $|A| \cdot |B|$ is not an integer multiple of $n-2$, then Lemma~\ref{lemma:divides+1} (see Remark~\ref{remark:lemma_common_k}, Item~\ref{itm:remark_2}) implies that $|P| \geq n-m_{even}+1$. Otherwise, all the rectangles of $P$ are balanced in rows and in columns in all blocks, and the number of ones that each of them covers is divisible by $n-2$.

Assume for the sake of contradiction that $|P| = n-m_{even}$. Since the total number of ones in $M$ is $n \cdot (n-2)$, it follows that at least $n-2 \cdot m_{even}$ of the rectangles of $P$ cover precisely $n-2$ ones. Indeed, otherwise the total number of ones covered by the rectangles of $P$ would be larger than
\[(n-2 \cdot m_{even}) \cdot (n-2) + m_{even} \cdot 2 (n-2) = n \cdot (n-2).\]
Since $n-2$ is odd, it follows that each of those $n-2 \cdot m_{even}$ rectangles has an odd number of rows and an odd number of columns. Since they are balanced in rows and in columns in the blocks of even size, it follows that the total number of rows of every such rectangle in the blocks of odd size is odd, and similarly,  the total number of their columns in these blocks is odd. In particular, each of these $n-2 \cdot m_{even}$ rectangles covers at least one entry whose both row and column belong to a block of odd size. Since there are $n_{odd} \cdot (n_{odd}-2)$ ones in such entries, it follows that $n-2 \cdot m_{even} \leq n_{odd} \cdot (n_{odd}-2)$, a contradiction.
\end{proof}

We finally turn to prove Theorem~\ref{thm:2-regular} and to determine for every integer $r$, the smallest possible binary rank of the complement of a $2$-regular matrix with binary rank $r$. The proof uses Theorems~\ref{thm:2_even} and~\ref{thm:2_even_odd}.

\begin{proof}[ of Theorem~\ref{thm:2-regular}]
For an integer $r \geq 4$, let $M$ be a $2$-regular $0,1$ matrix with $\Rbin(M) = r$.
Since applying permutations to the rows and columns of a matrix does not change its binary rank, it can be assumed, by Lemma~\ref{lemma:2-reg-structure}, that $M$ is a $(2; n_1,\ldots,n_m)$ circulant block diagonal matrix for integers $m \geq 1$ and $n_1, \ldots, n_m \geq 2$. Note that $M$ is an $n \times n$ matrix for $n = \sum_{i=1}^{n}{n_i}$.

Let $m_{even}$ denote the number of indices $i \in [m]$ for which $n_i$ is even.
Since $M$ is not the all-one $2 \times 2$ matrix, Lemma~\ref{lemma:2-reg-rankR} implies that
\begin{eqnarray}\label{eq:real_rank_M}
\Rreal(M) = \Rreal(\overline{M}) = n-m_{even}.
\end{eqnarray}
Let $m_2$ denote the number of indices $i \in [m]$ for which $n_i=2$, and let $m_{l}$ denote the number of indices $i \in [m]$ for which $n_i$ is even and larger than $2$, that is, $m_{l} = m_{even}-m_2$.
Recall, using Lemma~\ref{lemma:isolation}, that the binary rank of $D_{n_i,n_i-2}$ is equal to $n_i$ whenever $n_i > 2$ and to $1$ whenever $n_i=2$.
This implies that the binary rank of $M$ satisfies $\Rbin(M) = r = n - m_2$.

Consider first the case where $n_i$ is even for all $i \in [m]$.
If $n_i=2$ for all $i \in [m]$, then it follows from~\eqref{eq:real_rank_M} that $\Rbin(\overline{M}) \geq \Rreal(\overline{M}) = r$, so the required bound is clearly satisfied in this case for all $r \geq 4$.
Otherwise, observe that $n \geq 2 \cdot m_2 + 4 \cdot m_{l}$, and that this implies that
\[ r = n-m_2 \geq m_2 + 4 \cdot m_{l} \geq 4 \cdot m_{l}.\]
Since $n_i >2$ for some $i \in [m]$, we can apply Theorem~\ref{thm:2_even} to obtain that
\[\Rbin(\overline{M}) =n-m_{even}+1 = n-m_2-m_l+1 = r-m_l+1 \geq \frac{3r}{4}+1.\]
Since the binary rank of a matrix is an integer, the desired bound follows.

It remains to consider the case where there exists an $i \in [m]$ for which $n_i$ is odd and thus satisfies $n_i \geq 3$.
Using~\eqref{eq:real_rank_M}, we obtain that
\[\Rbin(\overline{M}) \geq \Rreal(\overline{M}) = n-m_{even} = n-m_2-m_l = r-m_l.\]
Observe that $n \geq 2 \cdot m_2 + 4 \cdot m_{l}+3$, which implies that $r = n-m_2 \geq m_2 + 4 \cdot m_{l}+3$.
If this inequality is strict, that is, $r \geq m_2 + 4 \cdot m_{l}+4$, then
\[ \Rbin(\overline{M}) \geq r-m_l \geq r - \frac{r-m_2-4}{4} \geq \frac{3r}{4}+1,\]
and the desired bound follows.
Otherwise, it holds that $r= m_2 + 4 \cdot m_{l}+3$.
This implies, using Lemma~\ref{lemma:isolation}, that all the blocks of $M$ are of even size besides one of them whose size is $3$, and in particular it follows that $n$ is odd.
By Theorem~\ref{thm:2_even_odd}, applied with $m_{even} \geq 1$ and $n_{odd}=3$, it follows that if $n>5$ then it holds that
\[ \Rbin(\overline{M}) = n-m_{even}+1= r -m_l +1 = r- \frac{r-m_2-3}{4}+1 \geq \frac{3r}{4}+1,\]
as required. It remains to check the case of $n = 5$, where $M$ has two diagonal blocks that are equal to $D_{3,1}$ and $D_{2,0}$.
In this case we have $\Rbin(M) = \Rreal(M) = 4$, and thus $\Rbin(\overline{M}) \geq 4$, satisfying the claimed bound, and we are done.
\end{proof}

We conclude with the statement of the following corollary of Theorems~\ref{thm:gap_kIntro} and~\ref{thm:2-regular}.
\begin{corollary}\label{cor:2-regular}
For every integer $r \geq 4$, the smallest possible binary rank of the complement of a $2$-regular $0,1$ matrix with binary rank $r$ is
$\lceil\tfrac{3r}{4} \rceil+1$.
\end{corollary}

\section*{Acknowledgement}

We are grateful to the anonymous referee for valuable suggestions and comments.

\bibliographystyle{abbrv}
\bibliography{circulant_block}

\begin{thebibliography}{10}

\bibitem{BBGJK21}
K.~Balodis, S.~Ben{-}David, M.~G{\"{o}}{\"{o}}s, S.~Jain, and R.~Kothari.
\newblock Unambiguous {DNF}s and {A}lon-{S}aks-{S}eymour.
\newblock In {\em Proc. of the {IEEE} 62nd Annual Symposium on Foundations of
  Computer Science ({FOCS}'21)}, pages 116--124. IEEE, 2021.

\bibitem{Breuer69}
M.~A. Breuer.
\newblock Combinatorial equivalence of $(0, 1)$ circulant matrices.
\newblock {\em J. Comput. Syst. Sci.}, 3(1):8--23, 1969.

\bibitem{BrualdiMR86}
R.~A. Brualdi, R.~Manber, and J.~A. Ross.
\newblock On the minimum rank of regular classes of matrices of zeros and ones.
\newblock {\em J. Combin. Theory Ser. A}, 41(1):32--49, 1986.

\bibitem{GradyN97}
M.~Grady and M.~Newman.
\newblock The geometry of an interchange: {M}inimal matrices and circulants.
\newblock {\em Linear Algebra and its Applications}, 262:11--25, 1997.

\bibitem{GrahamP71}
R.~L. Graham and H.~O. Pollak.
\newblock On the addressing problem for loop switching.
\newblock {\em Bell Syst. Tech. J.}, 50(8):2495--2519, 1971.

\bibitem{Gregory89}
D.~A. Gregory.
\newblock Biclique partitions of the complement of a directed cycle.
\newblock {\em J. Comb. Math. \& Comb. Comp.}, 6:183--187, 1989.

\bibitem{GregoryPullman}
D.~A. Gregory and N.~J. Pullman.
\newblock Semiring rank: {B}oolean rank and nonnegative rank factorization.
\newblock {\em Journal of Combinatorics, Information and System Sciences},
  8(3):223--233, 1983.

\bibitem{HavivP22}
I.~Haviv and M.~Parnas.
\newblock On the binary and {B}oolean rank of regular matrices.
\newblock In {\em Proc. of the 47th International Symposium on Mathematical
  Foundations of Computer Science ({MFCS}'22)}, pages 56:1--56:14, 2022.

\bibitem{HefnerHLM90}
K.~A.~S. Hefner, T.~D. Henson, J.~R. Lundgren, and J.~S. Maybee.
\newblock Biclique coverings of bigraphs and digraphs and minimum semiring
  ranks of $\{0,1\}$-matrices.
\newblock {\em Congr. Numer.}, 71:115--122, 1990.

\bibitem{ingleton1956rank}
A.~W. Ingleton.
\newblock The rank of circulant matrices.
\newblock {\em J. London Math. Soc.}, s1--31(4):445--460, 1956.

\bibitem{JiangR93}
T.~Jiang and B.~Ravikumar.
\newblock Minimal {NFA} problems are hard.
\newblock {\em {SIAM} J. Comput.}, 22(6):1117--1141, 1993.
\newblock Preliminary version in ICALP'91.

\bibitem{KN97}
E.~Kushilevitz and N.~Nisan.
\newblock {\em Communication Complexity}.
\newblock Cambridge University Press, 1997.

\bibitem{MonsonPR95survey}
S.~D. Monson, N.~J. Pullman, and R.~Rees.
\newblock A survey of clique and biclique coverings and factorizations of
  $(0,1)$-matrices.
\newblock {\em Bull. Inst. Combin. Appl.}, 14:17--86, 1995.

\bibitem{Pullman88}
N.~J. Pullman.
\newblock Ranks of binary matrices with constant line sums.
\newblock {\em Linear Algebra and its Applications}, 104:193--197, 1988.

\bibitem{Schwartz20}
S.~Schwartz.
\newblock An overview of graph covering and partitioning.
\newblock Technical Report 20-24, ZIB, Takustr. 7, 14195 Berlin, 2020.

\end{thebibliography}

\end{document}